\documentclass[11pt]{amsart}
\usepackage{amsmath}
\usepackage{amsfonts}
\usepackage{tabularx,booktabs}
\usepackage{amsmath,tikz}
\usepackage{amsfonts}

\usepackage{amscd}
\usepackage{amsthm}
\usepackage{ esint }
\usepackage{amssymb} \usepackage{latexsym}
\usepackage{eufrak}
\usepackage{euscript}
\usepackage{epsfig}
\usepackage{graphics}
\usepackage{array}
\usepackage{enumerate}
\usepackage{dsfont}
\usepackage{color}
\usepackage{wasysym}
\usepackage{hyperref}
\usepackage{pdfsync}
\usepackage[linesnumbered,boxed,vlined,ruled]{algorithm2e}

\newcommand{\bel}[1]{\begin{equation}\label{#1}}

\newcommand{\be}{\begin{equation}}

\newcommand{\ba}{\begin{eqnarray}}
\newcommand{\ea}{\end{eqnarray}}

\newcommand{\qe}{\end{equation}}
\newcommand{\R}{{\mathbb R}}
\newcommand{\N}{{\mathbb N}}
\newcommand{\Z}{{\mathbb Z}}

\newcommand{\Aut}{{\mathrm Aut}^{D}}
\newcommand{\aut}{\mathrm Aut}

\newcommand{\ess}{{\mathrm ess}}
\newcommand{\I}{{\mathrm I}}

\newcommand{\Hmm}[1]{\leavevmode{\marginpar{\tiny%
$\hbox to 0mm{\hspace*{-0.5mm}$\leftarrow$\hss}%
\vcenter{\vrule depth 0.1mm height 0.1mm width \the\marginparwidth}%
\hbox to
0mm{\hss$\rightarrow$\hspace*{-0.5mm}}$\\\relax\raggedright #1}}}

\theoremstyle{theorem}
\newtheorem{thm}{Theorem}[section]
\newtheorem{satz}{Proposition}[section]
\theoremstyle{example}
\newtheorem{example}{Example}[section]
\theoremstyle{corollary}
\newtheorem{coro}{Corollary}[section]
\theoremstyle{lemma}
\newtheorem{lemma}{Lemma}[section]
\theoremstyle{definition}
\newtheorem{defi}{Definition}[section]
\theoremstyle{proof}

\theoremstyle{remark}
\newtheorem{rem}{Remark}[section]

\begin{document}

\title[Dirichlet $p$-Laplacian eigenvalues and Cheeger constants]{Dirichlet $p$-Laplacian eigenvalues and Cheeger constants on symmetric graphs}

\author{Bobo Hua}
\address{School of Mathematical Sciences, LMNS, Fudan University, Shanghai 200433, China.}
\email{bobohua@fudan.edu.cn}

\author{Lili Wang}
\address{School of Mathematical Sciences, Fudan University, Shanghai 200433, China.}
\email{lilyecnu@outlook.com}
\begin{abstract}
{In this paper, we study eigenvalues and eigenfunctions of $p$-Laplacians with Dirichlet boundary condition on graphs. We characterize the first eigenfunction (and the maximum eigenfunction for a bipartite graph) via the sign condition. By the uniqueness of the first eigenfunction of $p$-Laplacian, as $p\to 1,$ we identify the Cheeger constant of a symmetric graph with that of the quotient graph. By this approach, we calculate various Cheeger constants of spherically symmetric graphs.}
\end{abstract}
\maketitle
\section{Introduction}
The spectrum of the Laplacian on a domain in the Euclidean space was extensively studied in the literature, see e.g. \cite{Courant-Hilbert53, Reed-Simon78}. There were many far-reaching generalizations on Riemannian manifolds, see \cite{Chavel84, SchoenYau94}.

In 1970, Cheeger \cite{Cheeger70} introduced an isoperimetric constant, now called Cheeger constant, on a compact manifold to estimate the first non-trivial eigenvalue of the Laplace-Beltrami operator, see also \cite{Yau75,Kawohl-Fridman03}. A graph is a combinatorial structure consisting of vertices and edges. Cheeger's estimate was generalized to graphs by Alon-Milman \cite{Alon-Milman85} and Dodziuk \cite{Dodziuk84}, respectively. Inspired by these results, there were many Cheeger type estimates on graphs, see e.g. \cite{Dodziuk-Kendall86, Lubotzky94, Fujiwara96,  LeeGharanTrevisan12, BauerHuaJost14, Liu15, Bauer-Keller-Wojciechowski15,Keller-Mugnolo16,Tudisco-Hein18}.  It turns out that Cheeger's estimates on graphs are useful in computer sciences \cite{Donath-Hoffman73, Ng-Jordan-Weiss2001, Bolla13}.

As elliptic operators, the $p$-Laplacians are nonlinear generalizations of the Laplacian in Euclidean spaces and Riemannian manifolds. The spectral theory of the $p$-Laplacians were studied by many authors, to cite a few \cite{Lindqvist93, Matei00,Wang12,Valtorta12,andrewsclutterbuck,NaberValtorta14,SetoWei17}.  Yamasaki \cite{Yamasaki79} proposed a discrete version of $p$-Laplacian on graphs. The spectral theory for discrete $p$-Laplacians was studied by \cite{Amghibech03,Takeuchi03,Hein-Buhler10,Chang-Shao-Zhang15,Keller-Mugnolo16}. It was well-known that the Cheeger constant is equal to the first eigenvalue of $1$-Laplaican, i.e. a Sobolev type constant, see  \cite{FedererFleming60,ChengOden97,Chung97,Li12, Chang-Shao-Zhang15,Chang16,Keller-Mugnolo16, Chang-Shao-Zhang17}. So that Cheeger's estimate reveals a connection between the first eigenvalues of $p$-Laplacians for $p=1$. In this paper, we study the spectral theory of $p$-Laplacians on graphs, and use the limit $p\to 1$ to investigate the Cheeger constant.

We recall the setting of weighted graphs. Let $(V,E)$ be a locally finite, simple, undirected graph. Two vertices $x,y$ are called neighbours, denoted by $x\sim y$, if there is an edge connecting $x$ and $y,$ i.e. $\{x,y\}\in E.$ Let $\mu: E\to \R_+, \{x,y\}\to\mu_{xy}=\mu_{yx},$ be the edge weight function. We extend $\mu$ to $V\times V$ by setting $\mu_{xy}=0$ if $\{x,y\}\notin E.$
Let $\nu: V\to \R_+, x\mapsto \nu_x,$ be the vertex measure. The weights $\mu$ and $\nu$ can be regarded as discrete measures on $E$ and $V$ respectively. We call the quadruple $G=(V,E,\nu,\mu)$ a weighted graph. For any subset $\Omega\subset V,$ $p\in [1,\infty),$ we denote by $\|u\|_{p,\Omega}:=\left(\sum_{x\in \Omega}|u(x)|^p\nu_x\right)^{\frac1p}$ the $\ell^p$ norm of a function $u$ on $\Omega.$

For a weighted graph $G$ and a finite subset $\Omega\subset V,$ we define the $p$-Laplacian, $p\in(1,\infty)$, with Dirichlet boundary condition on $\Omega.$ We denote by $\R^S$ the set of functions on $S\subset V.$ For any function $u\in \R^\Omega,$ the null-extension of $u$ is denoted by $\overline{u}\in \R^V,$ i.e. $\overline{u}(x)=u(x),$ $x\in \Omega,$ and $\overline{u}(x)=0,$ otherwise.  The $p$-Laplacian with Dirichlet boundary condition, Dirichlet $p$-Laplaican in short, on $\Omega$ is defined as
\begin{align}\begin{split}\label{Dir-p-Lap}
\Delta_p u (x):= \frac{1}{\nu_x}\sum_{y\in V}\mu_{xy}|\bar{u}(y)-\bar{u}(x)|^{p-2}\left(\bar{u}(y)-\bar{u}(x)\right), \forall x\in \Omega.
\end{split}\end{align}
We say $f\in \R^\Omega$ is an \emph{eigenfunction} (or eigenvector) pertaining to the eigenvalue $\lambda$ for the Dirichlet $p$-Laplacian on $\Omega$ if \begin{equation}\label{eigen-equation}
\Delta_{p,\Omega}f=-\lambda |f|^{p-2}f \ \ \mathrm{on}\ \Omega\quad \mathrm{and}\ f\not\equiv 0.
\end{equation} 

For any $p\in [1,\infty),$ let
$E_p:\R^{\Omega}\to \R$ be the $p$-Dirichlet functional on $\Omega,$ defined as
\begin{align}\begin{split}\label{energy-func}
E_p(u):=\frac{1}{2}\sum_{x,y\in V, x\sim y}|\bar{u}(y)-\bar{u}(x)|^p\mu_{xy}.
\end{split}\end{align} As is well-known, for $p>1,$ $f$ is an eigenfunction for the Dirichlet $p$-Laplacian on $\Omega$ if and only if $\frac{f}{\|f\|_{p,\Omega}}$ is a critical point of the functional $E_p(u)$
under the constraint $\|u\|_{p,\Omega}=1, \ u\in \R^{\Omega}.$ The critical point theory for the case $p=1$ is subtle, see e.g. \cite{Hein-Buhler10, Chang16}, for which the operator is called $1$-Laplacian. Note that the $p$-Laplacian depends on the weights $\mu$ and $\nu.$ If we choose $\nu_x=\sum_{y\sim x}\mu_{xy},$ $\forall x\in V,$ then the associated $p$-Laplacian is called normalized $p$-Laplacian. The $p$-Laplacian is a linear operator if and only if $p=2.$

In this paper, we are interested in the first eigenvalue (the maximum eigenvalue resp.), i.e. the smallest (largest resp.) eigenvalue, denoted by $\lambda_{1,p}(\Omega)$ ($\lambda_{m,p}(\Omega)$ resp.), and the associated eigenfunctions for $p$-Laplacians, $p\geq 1$.
By the well-known Rayleigh quotient characterization,
\begin{align}\begin{split}\label{eq:Rayleigh}
\lambda_{1,p}(\Omega)=\inf_{u\in \R^\Omega, u\not\equiv 0}\frac{E_p(u)}{\|u\|^p_{p,\Omega}}, \quad \lambda_{m,p}(\Omega)=\sup_{u\in \R^\Omega, u\not\equiv 0}\frac{E_p(u)}{\|u\|^p_{p,\Omega}}, \quad p\geq 1.
\end{split}\end{align}

Analogous to the continuous case \cite{Kawohl-Lindqvist06}, we obtain the following characterization of first eigenfunctions. A finite subset $\Omega\subset V$ is called connected if the induced subgraph on $\Omega$ is connected, i.e. for any two vertices in $\Omega$ there is a path in the induced subgraph on $\Omega$ connecting them.
\begin{thm}\label{First-eigen}
Let $G=(V,E,\nu,\mu)$ be a weighted graph and $\Omega$ be a finite connected subset of $V.$
Then the eigenfunction $f$ of Dirichlet $p$-Laplacian on $\Omega$ is a first eigenfunction if and only if either $f>0$ on $\Omega,$ or $f<0$ on $\Omega$.
Moreover, the first eigenfunction is unique up to the constant multiplication.
\end{thm} The first eigenfunction can be characterized via the fixed-sign condition. The uniqueness of the first eigenfunction will be crucial for our applications.

For any $U\subset V,$ we denote by $\partial U:=\{\{x,y\}\in E: x\in U, y\in V\setminus U \}$ the edge boundary of $U.$
The (Dirichlet) Cheeger constant on a finite subset $\Omega$ is defined as
\begin{equation}\label{eq:Cheegerconstant}h_{\mu,\nu}(\Omega)=\min_{\emptyset \neq U\subset \Omega}\frac{|\partial U|_\mu}{|U|_\nu},\end{equation} where $|\partial U|_\mu=\sum_{\{x,y\}\in\partial U}\mu_{xy}$ and $|U|_\nu=\sum_{x\in U}\nu_x$.
See Definition~\ref{def:Cheegerinf} for Cheeger constants, $h(G)$ and $h_{\infty}(G),$ of infinite graphs. The subset $U$ attains the minimum in \eqref{eq:Cheegerconstant} is called a Cheeger cut of $\Omega.$  For a finite graph without boundary the Cheeger constant was proven to be equal to the first nontrivial eigenvalue of $1$-Laplacian, see e.g. \cite[Proposition~4.1]{Hein-Buhler10} and \cite[Theorem~{5.15}]{Chang16}. The following is an analogous result for the Dirichlet boundary case.
\begin{satz}\label{c:1-Lap} Let $\Omega$ be a finite subset of $V.$ Then
\[\lambda_{1,1}(\Omega)=h_{\mu,\nu}(\Omega),\]
where $\lambda_{1,1}(\Omega)$ is the first eigenvalue of Dirichlet $1$-Laplacian.
\end{satz}



For the linear normalized Laplacian on a finite graph $(V,E)$ without boundary, the maximum eigenvalue can be used to characterize the bipartiteness of the graph. Recall that a graph is called bipartite if its vertex set can be split into two subsets $V_1,V_2$ such that every edge connects a vertex in $V_1$ to one in $V_2.$ As is well-known \cite{Chung97}, the maximum eigenvalue is $2$ if and only if the graph is bipartite. More importantly, there is an involution $S:\R^V\to \R^V,$
\begin{equation}\label{absolu-op}
S(u)(x)=\left\{\begin{array}{cl}
u(x),&x\in V_1,\\
-u(x),&x\in V_2,
  \end{array}\right.\end{equation} which transfers an eigenfunction $u$ of eigenvalue $\lambda$ to an eigenfunction $S(u)$ of eigenvalue $2-\lambda.$ Similar results hold for linear normalized Laplacians with Dirichlet boudnary condition, see \cite{BauerHuaJost14}. By this result, one easily figure out the sign condition for the maximum eigenfunction via that of the first eigenfunction. However, the involution $S$ doesn't work well for the nonlinear case, i.e. $p\neq 2,$ see e.g. Example \ref{conter-exam}. By using a convexity argument, we circumvent the difficulty and give the characterization of maximum eigenfunction by the sign condition for bipartite subgraphs.
\begin{thm}\label{maxi-funct}
Let $G=(V,E,\nu,\mu)$ be a weighted graph and $\Omega$ be a finite connected bipartite subgraph of $G$. Assume $f$ is an eigenfunction of Dirichlet $p$-Laplacian on $\Omega$. Then $f$ is a maximum eigenfunction if and only if $f$ satisfies $f(x)f(y)<0$ for $x\sim y$ and $x,y\in \Omega$.
Moreover, the maximum eigenfunction is unique up to the constant multiplication.
\end{thm}

We study Cheeger constants on symmetric graphs. For a weighted graph $G=(V,E,\nu,\mu),$ an automorphism of $G$ is a graph isomorphism $g:V\to V$ satisfying
\[
\mu_{g(x)g(y)}=\mu_{xy},\nu_{g(x)}=\nu_x, \forall x,y \in V.\] The set of automorphisms of $G$ form a group, denoted by $\aut(G).$
For our purposes, we say that an infinite graph $G$ is ``symmetric" if there is a subgroup $\Gamma$ of the automorphism group acting on $G$ finitely, i.e. each orbit for the action of the group $\Gamma,$ called $\Gamma$-orbit, consists of finitely many vertices. For any $x\in V,$ we denote by $[x]$ the $\Gamma$-orbit of $x.$
We define the quotient graph $G/\Gamma$ as follows: The set of vertices consists of the $\Gamma$-orbits; two different orbits $[x],[y]$ are adjacent if there are $x'\in [x],y'\in [y]$ such that $x'\sim y',$ and the edge weight is defined as $$\quad\mu_{[x][y]}=\sum_{x_1\in [x],y_1\in [y]}\mu_{x_1y_1};$$ the vertex weight is defined as $\nu_{[x]}=\sum_{x_1\in [x]}\nu_{x_1}.$ Note that in our definition, the quotient graph $G/\Gamma$ has no self-loops, although there could be edges between vertices in one orbit in $G$.

\begin{thm}\label{quotient-graph-Chee}
Let $G=(V,E,\nu,\mu)$ be a weighted graph, and $\Gamma$ be a subgroup of the automorphism group $\aut(G)$ which acts finitely on $G.$ Then
$$h(G)=h(G/\Gamma),\ \ h_{\infty}(G)=h_{\infty}(G/\Gamma).$$
\end{thm}
\begin{rem} \begin{enumerate}\item This theorem yields that we can reduce the computation of Cheeger constants of $G$ to that of the quotient graph.
 \item Note that the Cheeger cuts of a graph are usually not unique, see e.g. Example \ref{cut-non-unique}.
  But the proof of theorem indicates that among them there is one Cheeger cut consisting of
  $\Gamma$-orbits.
So that, for the computation of Cheeger constants of a symmetric graph one can treat the orbits as integrality.
  \end{enumerate}
\end{rem}


There is a natural metric on the graph $G$, the combinatorial distance $d$, defined as
$d(x,y)=\inf\{k: \exists x=x_0\sim \cdots \sim x_k=y\}$, i.e. the length of the shortest path connecting $x$ and $y$ by assigning each edge the length one.
For the combinatorial distance $d$ of the graph,
we denote by $B_r(x)=\{y\in V: d(y,x)\leq r\}$ the ball of radius $r$ centered at $x\in V$ and by $S_r(x):=\{y\in V: d(y,x)=r\}$ the $r$-sphere centered at $x$. We call a graph $G$ is \textbf{spherically symmetric} centered at a vertex $x_0$ if for any $x,y\in S_r(x_0)$ and $r\in \mathbb N\cup \{0\}$, there exits an automorphism of $G$ which leaves $x_0$ invariant and maps $x$ to $y$, see \cite{Keller-Lenz-Wojciechowski13,Breuer-Keller13}.
Then there is an associated  subgroup $\Gamma\leq \aut(G)$ acts finitely on $G$ such that the $\Gamma$-orbits are exactly $\{S_r(x_0)\}_{r=0}^\infty$. In this case, the quotient graph $G/\Gamma$ is a ``one dimensional" model. In $G/\Gamma$, we denote by $\underline{B}_r$ the ball of radius $r$ centered at $[x_0],$ and by $\underline{A}_{r,R}=\underline{B}_R\setminus \underline{B}_{r}$ the annulus of inner radius $r$ and outer radius $R.$

\begin{thm}\label{t:spherically symmetric}
Let $G$ be a spherically symmetric graph centered at $x_0\in V$
with the associated subgroup $\Gamma$ of the automorphism group. Then
\begin{equation*}\label{e:Cheeger radial}
h(G)=\inf_{r\geq 0}\frac{|\partial \underline{B}_r|}{|\underline{B}_r|},
\end{equation*}
\begin{equation*}\label{e:radial symmetry eq3}
h_{\infty}(G)=\liminf_{r\to \infty}\inf_{R\geq r+1}\frac{|\partial \underline{A}_{r,R}|}{|\underline{A}_{r,R}|}.
\end{equation*}
Moreover, if $G$ has infinite $\nu$-measure, then
\begin{equation*}\label{e:reduction to balls}
h_{\infty}(G)=\liminf_{r\to \infty}\frac{|\partial \underline{B}_r|}{|\underline{B}_r|}.
\end{equation*}
\end{thm}

The paper is organized as follows. The basic set up and concepts introduced in \S 2.
In \S 3, we prove the sign characterization of first eigenfunctions, Theorem \ref{First-eigen}.
In \S4, we prove Theorem \ref{maxi-funct}, the sign characterization of maximum eigenfuntions for bipartite subgraphs.
In \S 5, using the analytic approach, we identify the Cheeger constant of a symmetric graph with that of the quotient graph, Theorem \ref{quotient-graph-Chee}.
In \S 6, we introduce a ``one dimensional" model graph as the quotient graph of a spherically symmetric graphs, and prove Theorem \ref{t:spherically symmetric}.
In Appendix, we calculate various Cheeger constants of spherically symmetric graphs, for example, Fujiwara's spherically symmetric trees in Appendix \ref{exam1} and Wojciechowski's anti-trees in Appendix \ref{exam2}.

\section{Preliminary}


For a weighted graph $G=(V,E,\nu,\mu)$ and a finite subset $\Omega\subset V,$ we define the $p$-Laplacian, $p\in(1,\infty)$, with Dirichlet boundary condition on $\Omega.$ We denote by $\R^S$ the set of functions on $S\subset V.$ For any function $u\in \R^\Omega,$ the null-extension of $u$ defined as
\begin{align}\begin{split}\label{extension}
\bar{u}(x)
=\left\{\begin{array}{ll}
u(x),& x\in \Omega,\\
0,& x\in V\setminus\Omega.
\end{array}\right.
\end{split}\end{align}
Throughout the paper, we denote the null-extension of any function by $\bar{(\cdot)}$ in the paper.
The $p$-Laplacian with Dirichlet boundary condition, Dirichlet $p$-Laplaican in short, on $\Omega$ is defined as
\begin{align}\begin{split}\label{Dir-p-Lap}
\Delta_p u (x):= \frac{1}{\nu_x}\sum_{y\in V}\mu_{xy}|\bar{u}(y)-\bar{u}(x)|^{p-2}\left(\bar{u}(y)-\bar{u}(x)\right), \forall x\in\Omega.
\end{split}\end{align}


For any $p\in (1,\infty),$ we denote the $\ell^p$ space w.r.t. the measure $\nu$ by
\[
\ell^p_{\nu}(\Omega):=\Big\{u:\Omega\to \R: \sum_{x\in \Omega} |u(x)|^{p}\nu_x<\infty\Big\},
 \]
and denote the $\ell^p$ $\nu$-norm of a function $u$ by
\[
\|u\|_{p,\Omega}:=\Big(\sum_{x\in \Omega} |u(x)|^{p}\nu_x\Big)^{1/p}.
\]
The difference operator $\nabla$ is defined by $\nabla_{xy}u=u(y)-u(x)$ for any $x\sim y$. Then for any $u\in\R^V$, $|\nabla u|$ is a function on $E$ given by $|\nabla u|(\{x,y\})=|\nabla_{xy}u|$, $x\sim y$.
We denote the $\ell^p$ $\mu$-norm of the function $h\in \ell^p_{\mu}(E)$ by
\[
\|h\|_{p,E}:=(\sum\limits_{e\in E} |h(e)|^{p}\mu_e)^{1/p}.
\]
Then the $p$-Dirichlet functional $E_p$  defined as in (\ref{energy-func}) satisfies $E_p(u)=\|\nabla\bar{u}\|^p_{p,E}$.
Since the eigen-pair $(u,\lambda)$ of Dirichlet $p$-Laplacian satisfies eigenequation (\ref{eigen-equation}), by Green's formula, ref. \cite{Grigor09},
\begin{align}\begin{split}\label{Rayleigh quotient}
\lambda=\frac{\|\nabla\bar{u}\|^p_{p,E}}{\|u\|^p_{p,\Omega}}.
\end{split}\end{align}


For convenience,
we omit the subscript $\Omega$ if it is clear in the context, e.g. $\lambda_{1,p}:=\lambda_{1,p}(\Omega), \Delta_p u:=\Delta_{p,\Omega}u$ and so on.


\section{First eigenfunctions and eigenvalues to Dirichlet p-Laplacians}
In this section, we give an equivalent characterization for
first eigenfucntions of Dirichlet $p$-Laplacian.
Firstly, we prove the following lemma.
\begin{lemma}\label{nonzero}
Let $G=(V,E,\nu,\mu)$ be a weighted graph and $\Omega\subset V$ be a finite connected subset.
Assume $u\in \mathbb R^\Omega$ is an eigenfunction of Dirichlet $p$-Laplacian on $\Omega$.
If $u\geq 0$ ($u\leq 0$, resp.) on $\Omega$, then $u>0$ ($u<0$ resp.) on $\Omega$.
\end{lemma}
\begin{proof}
We show this by contradiction. Suppose $u\geq 0$ and there exists $x_0\in \Omega$ such that $u(x_0)=0$.
By (\ref{eigen-equation}), we have
\begin{align}\begin{split}\label{zero-point-Lap-equ}
0=-\lambda u^{p-1}(x_0)
=\Delta_p u(x_0)
=\frac{1}{\nu_{x_0}}
\sum\limits_{y\in V}|\bar{u}(y)|^{p-2}\bar{u}(y)\mu_{x_0y}.
\end{split}\end{align}
Hence, $\bar{u}(y)=0$ for any $y\sim x_0.$
By the connectedness of $\Omega$, $u\equiv 0$ on $\Omega.$
This contradicts to $u\not\equiv 0,$ since $u$ is an eigenfunction. Hence, $u>0$.

Replacing $u$ by $-u$ and using the same argument, we can show that $u\leq 0$ gives $u<0$.
\end{proof}
For any $\Omega\subset V$, we denote the vertex boundary of $\Omega$ by \[\delta\Omega:=\{y\in V\setminus \Omega: \exists x\in \Omega\ \mathrm{such\ that}\ y\sim x\}.\]

\begin{lemma}[\cite{Kim-Chung10, Park-Chung11}, Theorem A in \cite{Park11}]\label{l:comparison principle}
Let $u,v$ be functions on $\Omega\cup \delta\Omega$. Assume $u,v$ satisfy the following equation
\begin{align*}\begin{cases}
\Delta_p u(x)\geq \Delta_pv(x), \ \ &x\in\Omega,\\
u(x)=v(x)=0,                        &x\in\delta\Omega.
\end{cases}\end{align*}
Then $u \leq v $ on $\Omega$.
\end{lemma}
The comparison principle enables us to characterize the first eigenfunctions by their sign conditions.
\begin{thm}\label{first-eigen}
Let $\Omega$ be a finite connected subgraph of weighted graph $G=(V,E,\nu,\mu)$ and $u\in \mathbb R^\Omega$ be an eigenfunction of Dirichlet $p$-Laplacian on $\Omega$.
Then $u$ is the first eigenfunction if and only if either $u>0$ on $\Omega$ or $u<0$ on $\Omega$.
\end{thm}
\begin{proof}
We first show that the first eigenfunction $u$ satisfies either $u>0$ on $\Omega$ or $u<0$ on $\Omega$.
Let $u$ be a first eigenfunction pertaining to $\lambda_{1,p}$.
By scaling, w.l.o.g., we assume that $\|u\|_{p,\Omega}=1.$
By the Rayleigh quotient characterization (\ref{Rayleigh quotient}), we have
\begin{align}\begin{split}\label{bdd-by-first-eigens}
\lambda_{1,p}=\|\nabla \bar{u} \|^p_{p,E}\geq \|\nabla |\bar{u} |\|^p_{p,E}
\geq \lambda_{1,p}.
\end{split}\end{align}
Hence, above inequalities are equalities.
This implies that
\begin{align*}
0&=\|\nabla u\|^p_{p,E}-\|\nabla |u|\|^p_{p,E}\\
 &=\sum_{y\in V}\left(|\bar{u}(y)-\bar{u}(x)|^p-\left||\bar{u}|(y)-|\bar{u}|(x)\right|^p\right)\mu_{xy}
\geq 0,
\end{align*}
which implies that $\bar{u}(x)\bar{u}(y)\geq 0$ for $y\sim x$.
By (\ref{bdd-by-first-eigens}), $|\bar{u}|$ is a first eigenfunction of $\Omega$. By Lemma \ref{nonzero}, $|u(x)|>0, \forall x\in \Omega$. Hence, by the connectedness of $\Omega$, either $u> 0$ on $\Omega$ or $u<0$ on $\Omega$.

For another direction,
we choose $u\in \mathbb R^\Omega$ as the positive first eigenfunction pertaining to the first eigenvalue $\lambda_{1,p}$, replacing $u$ by $-u$ if $u<0$.
Using a contradiction argument, we assume that $u_0\in\mathbb R^\Omega$ is a positive eigenfunction of Dirichlet $p$-Laplacian pertaining to $\lambda$ with $\lambda>\lambda_{1,p}$.
Since $\Omega$ is finite, by scaling $u$ we may assume that $u(x)\leq u_0(x)$ for any $x\in \Omega.$
We claim that $u\leq \kappa u_0$ with
$\kappa=\Big(\frac{\lambda_{1,p}}{\lambda}\Big)^{\frac{1}{p-1}}<1.$ Note that on $\Omega$
\begin{eqnarray*}
  \Delta_p u&=&-\lambda_{1,p} u^{p-1}\geq -\lambda_{1,p} u_0^{p-1}\\
  &=&-\lambda(\kappa u_0)^{p-1}=\Delta_p(\kappa u_0),
\end{eqnarray*} and $\bar{u}|_{\delta\Omega}=0=\kappa \bar{u}_0|_{\delta \Omega}.$
The comparison principle for the $p$-Laplacian, Lemma \ref{l:comparison principle},
yields that $u\leq \kappa u_0$ on $\Omega.$ This proves the claim.

By the same argument, replacing $u_0$ by $\kappa u_0,\kappa^2 u_0,\cdots,$
one can show that $u\leq \kappa^n u_0$ on $\Omega$ for any $n\in \N.$
Taking the limit $n\to \infty,$
we get $u \equiv 0$ on $\Omega.$ This yields a contradiction. Hence, we obtain $\lambda=\lambda_{1,p}$ and $u_0$ is the first eigenfunction.
\end{proof}
\begin{lemma}\label{one-multi}
With the same assumption as in Theorem \ref{first-eigen},
 the first eigenfunction is unique (up to multiplication with constants).
\end{lemma}
\begin{proof}
Let $u_1,u_2$ be two first eigenfunctions of Dirichlet $p$-Laplacian.
It suffices to prove there exists a constant $c\neq 0$ such that $u_1=cu_2$.

By Theorem \ref{first-eigen}, either $u_i>0$ on $\Omega$ or $u_i<0$ on $\Omega$ for $i=1,2$.
We may assume $u_i>0$ on $\Omega$ and $\|u_i\|_{p,\Omega}=1$ for $i=1,2$.
Choose a new function $u=(u_1^p+u_2^p)^{1/p}.$ Then $\|u\|_p^p=2.$
Let $\bar{u}$, $\bar{u}_1$, $\bar{u}_2$ be the null-extension of $u$, $u_1$, $u_2$, respectively. We claim that
\begin{align}\begin{split}\label{p-norm-ineq}
|\nabla_{xy} \bar{u}|^p
\leq |\nabla_{xy} \bar{u}_1|^p+|\nabla_{xy} \bar{u}_2|^p, \ \ \forall x,y\in V,\ x\sim y.
\end{split}\end{align}
If $x\in \delta\Omega$ or $y\in \delta\Omega,$
then the equality holds trivially.
It's sufficient to prove the claim for the case $x,y\in \Omega.$
This follows from the convexity of $\ell^p$ norm,
denoted by $|\cdot|_p$, in $\R^2$.
Indeed, setting vectors $U=(u_1(x),u_2(x))$ and $V=(u_1(y),u_2(y)),$ we have
\begin{align*} 
|\nabla_{xy} u|^p
&=||U|_p-|V|_p|^p\leq |U-V|_p^p\\
&=|\nabla_{xy} u_1|^p+|\nabla_{xy} u_2|^p.
\end{align*}
By the strict convexity of $\ell^p$ norm for $1<p<\infty,$ the equality holds if and only if $U=cV$ for some $c>0.$
This proves the claim. 
By (\ref{p-norm-ineq}) and (\ref{Rayleigh quotient}), 
\begin{eqnarray*}
2\lambda_{1,p}=\lambda_{1,p}\|u\|_{p,\Omega}^p\leq\|\nabla \bar{u} \|_{p,E}^p\leq \|\nabla \bar{u}_1\|_{p,E}^p+\|\nabla \bar{u}_2\|_{p,E}^p
=2\lambda_{1,p}.
\end{eqnarray*}
Hence the above inequalities are in fact be equalities,
which implies that $\frac{u_1(x)}{u_2(x)}=\frac{u_1(y)}{u_2(y)}$
for any $x\sim y$ with $x,y\in \Omega.$ This yields that $u_1=cu_2$
on $\Omega$ by the connectedness of $\Omega.$ The proof is completed.
\end{proof}

\begin{proof}[Proof of Theorem~\ref{First-eigen}] The theorem follows from Theorem \ref{first-eigen} and Lemma \ref{one-multi}.
\end{proof}

It is well known that the first eigenvalue of $1$-Laplacian is given by the Cheeger constant (c.f. \cite{Hein-Buhler10, Chang16, Chang-Shao-Zhang17}). For completeness we give a proof for the Dirichlet $1$-Laplacian here.

\begin{proof}[Proof of Proposition \ref{c:1-Lap}]
Let $\mathds{1}_K(x)$ be characteristic function on $K$ defined by \begin{align*}
\mathds{1}_K(x)=
\begin{cases}1, \ \ &x\in K\\
0, &x\notin K.
\end{cases}
\end{align*}
For any function $f\in \R^\Omega$, let
$u=|f|$ and $\Omega_t(u):=\{x\in \Omega| u(x)>t\}$.
Set $\I_{x,y}=\big[\min\{u(x),u(y)\},\max\{u(x),u(y)\}\big)$.
Then $\{x,y\}\in\partial\Omega_t(u)$ if and only if $t\in \I_{x,y}$.
Hence,
\begin{align}\begin{split}\label{level-set-vol}
|\partial\Omega_t(u)|_{\mu}= \sum\limits_{\{x,y\}\in \partial\Omega_t(u)} \mu_{xy}
=\sum\limits_{\{x,y\}\in E} \mathds{1}_{\I_{x,y}}(t)\mu_{xy}.
\end{split}\end{align}

Since $\Omega_t\subset \Omega$, $h_{\mu,\nu}(\Omega)\leq h_{\mu,\nu}(\Omega_t)$.
By (\ref{level-set-vol}),
\begin{align}\begin{split}\label{co-area-formu}
\frac{1}{2}\sum_{x,y\in V}|\bar{u}(x)-\bar{u}(y)|\mu_{xy}
&=\sum\limits_{\{x,y\}\in E}\int_0^\infty \mathds{1}_{\I_{x,y}}(t)\mu_{xy}dt \\
&= \int_0^\infty\sum\limits_{\{x,y\}\in E} \mathds{1}_{\I_{x,y}}(t)\mu_{xy}dt \\
&=\int_0^\infty |\partial\Omega_t(u)|_{\mu}dt \\
&\geq h_{\mu,\nu}(\Omega)\int_0^\infty|\Omega_t(u)|_\nu dt.
\end{split}\end{align}
We also have $x\in\Omega_t(u)$ if and only if $\mathds{1}_{(t,\infty)}(u(x))=1$. Then
\begin{align}\begin{split}\label{area-form}
\int_0^\infty|\Omega_t(u)|_\nu dt
&=\int_0^\infty \sum_{x\in\Omega_t}\nu_x dt\\
&=\int_0^\infty \sum_{x\in\Omega}\mathds{1}_{(t,\infty)}(u(x))\nu_xdt\\
&=\sum_{x\in\Omega}\nu_x\int_0^\infty\mathds{1}_{(t,\infty)}(u(x))dt
 =\sum_{x\in\Omega}u(x)\nu_x=\|f\|_{1,\Omega}.
\end{split}\end{align}
Combining (\ref{co-area-formu}) and (\ref{area-form}), together with $E_1(u)=E_1(|f|)\leq E_1(f)$, we obtain
$\frac{E_1(f)}{\|f\|_{1,\Omega}}\geq h_{\mu,\nu}(\Omega)$.
Applying the Rayleigh quotient characterization \eqref{eq:Rayleigh}, we obtain $\lambda_{1,1}(\Omega)\geq h_{\mu,\nu}(\Omega)$.

On the other hand, let $U\subset \Omega$ be a Cheeger cut such that $h_{\mu,v}(\Omega)=\frac{|\partial U|_{\mu}}{|U|_\nu}$. Considering the characteristic function $\mathds{1}_U$,
by \eqref{eq:Rayleigh}, we have
\[
\lambda_{1,1}(\Omega)\leq \frac{E_1(\mathds{1}_U)}{\|\mathds{1}_U\|_{1,\Omega}}
 = \frac{\sum_{\{x,y\}\in \partial U}\mu_{xy}}{|U|_\nu}
 = \frac{|\partial U|_\mu}{|U|_\nu}=h_{\mu,\nu}(\Omega).
\]
Hence, we obtain $\lambda_{1,1}(\Omega)=h_{\mu,\nu}(\Omega)$. The proof is completed.
\end{proof}

In the rest of the section, we prove the monotonicity property of the first eigenvalue of Dirichlet $p$-Laplacian as $p$ varies, analogous to the continuous case. By mimicking the argument in \cite[Theorem~3.2]{Lindqvist93}, we prove the following result.
\begin{satz}\label{monotonicity}
 Let $\Omega$ be a finite connected subset of a weighted graph $G=(V,E,\nu,\mu)$ with $\nu_x=\sum\limits_{y\in V}\mu_{xy}$. For $1<p<s<\infty,$ we have
  $$p\lambda_{1,p}(\Omega)^{\frac1p}\leq s\lambda_{1,s}(\Omega)^{\frac1s}.$$
\end{satz}
\begin{proof}
Let $u$ be a first eigenfunction of Dirichlet $s$-Laplacian on $\Omega$
satisfying $u>0$ on $\Omega$ and $\|u\|_{s,\Omega}=1.$ Then
$\|\nabla \bar{u}\|_{s,E}^s=\lambda_{1,s}(\Omega).$
Let $h=u^{\frac{s}{p}}.$ Then $\|h\|_{p,\Omega}=1.$
By the Rayleigh quotient characterization (\ref{Rayleigh quotient}),
\begin{equation}\label{e:eq1}
p\lambda_{1,p}(\Omega)^{\frac1p}\leq p\|\nabla h\|_p.
\end{equation}
Let $S:=\{\{x,y\}\in E: \bar{h}(x)\neq \overline{h}(y)\}.$
For any given $\{x,y\}\in S$, by the symmetry $\{x,y\}=\{y,x\}$, we always assume that $\bar{u}(x)<\bar{u}(y)$,
\begin{eqnarray*}
|\nabla_{xy} \bar{h}|
&=&\bar{u}^{\frac{s}{p}}(y)-\bar{u}^{\frac{s}{p}}(x)
=\frac{s}{p}\int_{\bar{u}(x)}^{\bar{u}(y)}t^{\frac{s-p}{p}}dt\\
&=&\frac{s}{p}\fint_{\bar{u}(x)}^{\bar{u}(y)}t^{\frac{s-p}{p}}dt|\bar{u}(y)-\bar{u}(x)|,
\end{eqnarray*}
where $\fint_{\bar{u}(x)}^{\bar{u}(y)}$ denotes $\frac{1}{\bar{u}(y)-\bar{u}(x)}\int_{\bar{u}(x)}^{\bar{u}(y)}.$
Noting that $\frac{ps}{s-p}>1$ and $s>1,$ applying H\"{o}lder inequality, we obtain
\begin{eqnarray*}
\|\nabla \bar{h}\|_{p,S}^p&=&\Big(\frac{s}{p}\Big)^{p}\sum_{\{x,y\}\in S}
\Big(\fint_{\bar{u}(x)}^{\bar{u}(y)}t^{\frac{s-p}{p}}dt\Big)^p|\bar{u}(y)-\bar{u}(x)|^p\mu_{xy}\\
&\leq&\Big(\frac{s}{p}\Big)^{p}\Big(\sum_{\{x,y\}\in S}
\Big(\fint_{\bar{u}(x)}^{\bar{u}(y)}t^{\frac{s-p}{p}}dt\Big)^{\frac{ps}{s-p}}\mu_{xy}\Big)^{\frac{s-p}{s}}\\
& &\ \ \ \times\Big(\sum_{\{x,y\}\in S}|\bar{u}(y)-\bar{u}(x)|^s{\mu_{xy}}\Big)^{\frac{p}{s}}.
\end{eqnarray*}
From the H\"{o}lder inequality, we have
$$\Big(\fint_{\bar{u}(x)}^{\bar{u}(y)}t^{\frac{s-p}{p}}dt\Big)^{\frac{ps}{s-p}}\leq \fint_{\bar{u}(x)}^{\bar{u}(y)}t^sdt\leq \bar{u}^s(y).$$
Combining the above two inequalities, together with $\frac{s-p}{s}<1$,
we have
\begin{align}\begin{split}\label{eq2}
\|\nabla \bar{h}\|_{p,E}^p=\|\nabla \bar{h}\|_{p,S}^p
&\leq \Big(\frac{s}{p}\Big)^{p}\big(\sum_{\{x,y\}\in S}\bar{u}^s(y)\mu_{xy}\big)^\frac{s-p}{s}\big(\|\nabla \bar{u}\|_{s,S}^s\big)^\frac{p}{s}\\
&\leq  \Big(\frac{s}{p}\Big)^{p}\|u\|_{s,\Omega}^s\big(\|\nabla \bar{u}\|_{s,E}^s\big)^\frac{p}{s}\\
&= \Big(\frac{s}{p}\Big)^{p}\lambda^\frac{p}{s}_{1,s}(\Omega).
\end{split}\end{align}
Combining (\ref{e:eq1}) with (\ref{eq2}), we get the desired result.
\end{proof}

\section{Maximum eigenfunctions to Dirichlet p-Laplacians on bipartite subgraphs}
Recall that a graph is called a bipartite graph if its vertices can be divided into two disjoint sets $V_1$ and $V_2$ such that every edge connects a vertex in $V_1$ to one in $V_2$.  Vertex sets $V_1$ and $V_2$ are usually called the parts of the graph. In this section, we obtain an equivalent characterization for maximum eigenfunctions on a bipartite subgraph.

Let $S$ be the involution defined as (\ref{absolu-op}).
Now we give an example to show the relationship between first eigenfunctions and maximum eigenfunctions for Dirichlet $p$-Laplacian when $p\neq 2$.
\begin{figure}[htbp]
 \begin{center}
   \begin{tikzpicture}
    \node at (0,0){\includegraphics[width=0.6\linewidth]{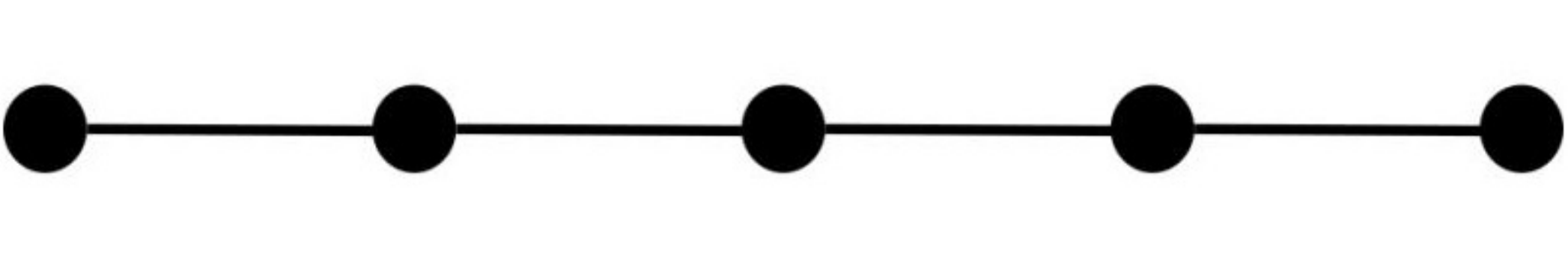}};
    \node at (-3.6, -.7){\Large $v_0$};
    \node at (-1.8, -.7){\Large $v_1$};
    \node at (0, -.7){\Large $v_2$};
    \node at (1.8, -.7){\Large $v_3$};
    \node at (3.6,   -.7){\Large $v_4$};
    \end{tikzpicture}
  \caption{\small }
\label{fig1}
 \end{center}
\end{figure}

\begin{example}\label{conter-exam}
Let $G=(V,E,\nu,\mu)$ be a weighted graph with $V=\{v_0,v_1,v_2,\\ v_3,v_4\}$ as shown in Figure \ref{fig1}. 
Assume $\mu_{v_iv_j}=1$ for $v_i\sim v_j$ and $ v_i, v_j\in V$, $\nu_{v_i}=\sum_{v_j\in V}\mu_{v_iv_j}$ for $0\leq i\leq 4$.
For $\Omega=\{v_1,v_2,v_3\}$,  by direct computation,  the first eigenfunction and maximum eigenfunction for $\Delta_p$, $p=4$, are
\[u_1(v_1)=0.422207, \ u_1(v_2)=0.966286,\ u_1(v_3)=0.422207,\]
and
\[ u_{\max}(v_1)=0.696725,\ u_{\max}(v_2)=-0.852721,\ u_{\max}(v_3)=0.696725,\]
where $\|u_1\|_{p,\Omega}=\|u_{\max}\|_{p,\Omega}=1$.
\end{example}

Obviously, $S(u_1)\neq u_{\max}$. This example indicates that there is no close  relation between first and maximum eigenfunctions for $p\neq 2$.

Next we describe the sign property of maximum eigenfunctions for $p$-Laplacian on a bipartite subgraph $\Omega$.

\begin{satz}\label{opp-sign}
Let $G=(V,E,\nu,\mu)$ be a weighted graph and $\Omega$ be a finite connected bipartite subgraph of $G$. If $f$ is a maximum eigenfunction of Dirichlet $p$-Laplacian on $\Omega$, then $f$ satisfies $f(x)f(y)<0$ for $x\sim y$ and $x,y\in \Omega$.
\end{satz}

\begin{proof}
Firstly we assume that $f$ is a maximum eigenfunction satisfying $\|f\|_{p,\Omega}=1$. It's sufficient to show that $f$ satisfies $f(x)f(y)<0$ for $x\sim y$ and $x,y\in \Omega$.
Let $h=S(|f|)$ with the involution $S$ defined as (\ref{absolu-op}), and $\bar{h}$, $\bar{f}$ be null-extension of $f$, $h$ defined as (\ref{extension}), respectively. Since $\Omega$ is a connected bipartite subgraph, then
\[\left(|\bar{f}(y)|+|\bar{f}(x)|\right)^p
=\left|\bar{h}(y)-\bar{h}(x)\right|^p, \forall x,y\in \Omega, x\sim y.\]
Hence, we have
\begin{align}\begin{split}\label{maxeigen-two-sides}
\lambda_{m,p}=E_p(f)&\leq\frac{1}{2}\sum\limits_{x,y\in V}\left(|\bar{f}(y)|+|\bar{f}(x)|\right)^p\mu_{xy}\\
&=\frac{1}{2}\sum\limits_{x,y\in V}\left|\bar{h}(y)-\bar{h}(x)\right|^p\mu_{xy}\\
&\leq \sup_{\begin{subarray}{l} 0\neq g\in \R^{\Omega} \\ \|g\|_{p,\Omega}=1\end{subarray}}
\|\nabla \bar{g}\|_{p,E}^p
=\lambda_{m,p}.
\end{split}\end{align}
Then the inequalities in (\ref{maxeigen-two-sides}) have to be equalities, which implies that
\begin{align}\begin{split}\label{max-eigenequation-bipart}
\lambda_{m,p}=\frac{1}{2}\sum\limits_{x,y\in V}\left|\bar{h}(y)-\bar{h}(x)\right|^p\mu_{xy}=\|\nabla \bar{h}\|_{p,E}^p
\end{split}\end{align}
and
\begin{align}\begin{split}\label{diff-sign}
f(x)f(y)\leq 0, \ \forall \ x,y\in\Omega, x\sim y.
\end{split}\end{align}
By (\ref{diff-sign}), it suffices to show that there is no vertex in $\Omega$ such that $f(x)=0$. We show this by contradiction. Suppose there is $x_0\in \Omega$ such that $f(x_0)=0$, then $h(x_0)=0$. By Green's formula (ref. \cite{Grigor09}), (\ref{max-eigenequation-bipart}) yields that $h$ is an eigenfunction satisfying eigen-equation (\ref{eigen-equation}). By (\ref{Dir-p-Lap}),
\begin{align}\begin{split}\label{zero-point}
0=-\lambda_{m,p} h^{p-1}(x_0)
=\Delta_p h(x_0)
=\frac{1}{\nu_{x_0}}
\sum\limits_{y\in \Omega}|h(y)|^{p-2}h
(y)\mu_{x_0y}.
\end{split}\end{align}
By $\Omega$ is a connected bipartite subgraph with bipartite parts $V_1, V_2$ and $h=S(|f|)$,
\begin{align}\begin{split}\label{key-ineq}
\sum\limits_{y\in \Omega}|h(y)|^{p-2}h
(y)\mu_{x_0y}=
\begin{cases}
-\sum\limits_{y\in V_2}|f|^{p-1}(y)\mu_{x_0y},\ \ x_0\in V_1,\\
\sum\limits_{y\in V_1}|f|^{p-1}(y)\mu_{x_0y},\ \ x_0\in V_2.
\end{cases}
\end{split}\end{align}
Combining (\ref{zero-point}) and (\ref{key-ineq}), we obtain $f(y)=0$ for $y\in \Omega$ and $y\sim x_0$. Since $\Omega$ is connected, $f\equiv 0$, that is, $h\equiv 0$. This contradict to $h\not\equiv 0$ since $h$ is an eigenfunction.
We get the desired result.
\end{proof}
To prove the other direction of Theorem \ref{maxi-funct}, we need the some lemmas.

We write
$\Omega=\{v_1,v_2,\cdots, v_N\}$ and the vertex boundary \[\delta\Omega=\{v_{N+1},\cdots, v_{N+b}\}.\]
For any function $u:\R^\Omega\to \R$, let $\bar{u}_i:=\bar{u}(v_i)$.
For simplicity, we write $\mu_{ij}=\mu_{v_iv_j}, \nu_i=\nu_{v_i}$ for $1\leq i,j\leq N+b$.
Then the conditions $\|u\|^p_{p,\Omega}=1$ and $u\big|_{\delta\Omega}=0$ are given by
$\sum_{1\leq i\leq N}|u_i|^p\nu_i=1$ and $u_{N+1}=\cdots=u_{N+b}=0$, respectively.
Hence,
$\bar{u}$ restricted $\Omega\cup\delta\Omega$ corresponds to a vector $\left(u_1, \cdots, u_N,0\cdots, 0\right)\in \R^{N+b}$.

Let $S_1^p:=\left\{u\in \R^{N+b}\Big|\sum_{1\leq i\leq N}|u_i|^p\nu_i=1, u_{N+1}=\cdots=u_{N+b}=0\right\}.$ Then
the eigenvalue problem of Dirichlet $p$-Laplacian on $\Omega$ is to find the critical values of the functional
\begin{equation}\label{eq:critical point Ep1}
E_p: S_1^p\to \R,\ \
E_p(u)=\frac12\sum_{1\leq i, j\leq N+b}|u_i-u_j|^p\mu_{ij}.
\end{equation}
We denote 
\[\R_+^N:=\left\{(x_1,x_2,\cdots,x_N,0,\cdots,0)\in\R^{N+b}\Big| x_i>0, 1\leq i\leq N\right\}\]
and define another variational problem $K_p: A_1\to \R$ on $\R_+^N$ as follows.
\begin{align}\begin{split}\label{eq:critical point Ep}
K_p(h):=\frac12\sum\limits_{1\leq i,j\leq N+b}|h_i+h_j|^p\mu_{ij},\  A_1:=\Big\{h\in \R^N_+|\sum\limits_{1\leq i\leq N}h_i^p\nu_i=1\Big\}.
\end{split}\end{align}

\begin{defi}
Let $E_i: S_i\rightarrow \mathbb R$, $i=1,2$, be smooth functionals on smooth manifolds $S_i$.
We say that $E_1$ and $E_2$ are equivalent under the map $P: S_1\rightarrow S_2$ if
\begin{enumerate}
\item $P$ is a diffeomorphism between $S_1$ and $S_2$;
\item $u$ is a critical point of $E_1$ with the critical value $\lambda$ if and only if $P(u)$ is a critical point of $E_2$ with the critical value $\lambda$.
\end{enumerate}
\end{defi}

\begin{lemma}\label{Equi-12}
Let $\Omega$ be a bipartite subgraph of the weighted graph $G$ with the parts $V_1$ and $V_2$. Set
\[
S_1':=\left\{u\in S_1^p| u_i=u(v_i)>0\ \textrm{for} \ v_i\in V_1, u_iu_j<0\ \textrm{for}\ \mu_{ij}>0\right\}.
\]
Then the variational problem (\ref{eq:critical point Ep1}) on $S_1'$ and the variational problem (\ref{eq:critical point Ep}) are equivalent under the map
$\bar{S}: S_1'\to A_1$ defined as
\begin{align*}
h_i=\bar{S}(u)(v_i):=
\begin{cases}
u(v_i), \ \ & v_i\in V_1\\
-u(v_i), &v_i\in V_2\\
0,    &v_i\in\delta\Omega.
\end{cases}
\end{align*}
\end{lemma}
\begin{proof}
Obviously, $\bar{S}: S_1'\to A_1$ is a one-to-one mapping and its inverse mapping is $\bar{S}^{-1}=\bar{S}$.
For any $u\in S_1'$,  set $h=\bar{S}(u)$, this implies that $|u_i-u_j|^p\mu_{ij}=|h_i+h_j|^p\mu_{ij}$ for any $1\leq i,j\leq N+b$.
Hence, we have $E_p(u)=K_p(h).$
We also obtain $h_i=|u_i|$ from $u_iu_j<0\ \textrm{for}\ \mu_{ij}>0$,  then the constraint condition $\sum\limits_{1\leq i\leq N}|u_i|^p\nu_i=1$
is equivalent to $\sum\limits_{1\leq i\leq N}h_i^p\nu_i=1.$

Using Lagrange multiplier method, together with $\frac{\partial}{\partial u_i}=c\frac{\partial}{\partial h_i}$ for $1\leq i\leq N$ with $c=\pm1$,
we obtain the critical value of the functional $E_p$ on $S_1'$ is same as the critical value of the functional $K_p(u)$ on $A_1$.
\end{proof}

We introduce a new variational problem $Q_p: P_1\to \R$ as follows.
\begin{align}\begin{split}\label{e:reformulation}
 Q_p(g):=\frac12\sum\limits_{1\leq i, j\leq N}|g_i^{\frac1p}+g_j^{\frac1p}|^p\mu_{ij}, \ \
 P_1:=\Big\{g\in \R^N_+|\sum\limits_{1\leq i\leq N}g_i\nu_i=1\Big\}.
\end{split}
\end{align}

\begin{lemma}\label{Equivalence}
The variational problem (\ref{eq:critical point Ep}) and (\ref{e:reformulation}) are equivalent under the map:
\[
T: A_1\to P_1,\ \  T(h)=h^p.
\]
\end{lemma}
\begin{proof}
It's easy to observe that $T$ is a bijective and $T^{-1}: P_1\to A_1,\ \  T^{-1}(g)=g^\frac{1}{p}. $
By Lagrange multiplier method, it suffices to show that
\begin{align}\begin{split}\label{K-Lag-mul}
\frac{\partial }{\partial h_i}\big[K_p(h)
    +\lambda(1-\sum_{1\leq i\leq N}h_i^p\nu_i)\big]=0, \ \ 1\leq i\leq N
\end{split}\end{align}
is equivalent to
\begin{align}\begin{split}\label{Q-Lag-mul}
\frac{\partial}{\partial g_i}\big[Q_p(g)+\lambda\big(1-\sum\limits_{1\leq i\leq N}g_i\nu_i\big)\big]=0,\ \ 1\leq i\leq N
\end{split}\end{align}
for some $\lambda\neq 0$.

On one hand, by (\ref{eq:critical point Ep}), (\ref{e:reformulation}) and  $g=T(h)=h^p$, we obtain
$Q_p(g)=Q_p(T(h))=K_p(h)$. Combining this with $\frac{\partial }{\partial g_i}=\frac{1}{ph_i^{p-1}}\frac{\partial}{\partial h_i}$ for $1\leq i\leq N$, we obtain that (\ref{K-Lag-mul}) yields (\ref{Q-Lag-mul}).
On the other hand, by (\ref{eq:critical point Ep}), (\ref{e:reformulation}) and  $h=T^{-1}(g)=g^\frac{1}{p}$,
$K_p(h)=K_p\left(T^{-1}(g)\right)=Q_p(g^\frac{1}{p})$. Using
$\frac{\partial}{\partial h_i}=
p g_i^{1-\frac{1}{p}}\frac{\partial}{\partial g_i}$, we obtain
that (\ref{Q-Lag-mul}) yields (\ref{K-Lag-mul}).
This proves the lemma.
\end{proof}



\begin{lemma}\label{lem:concavity}
  $Q_p:\R^N_+\to \R$ is a concave function.
\end{lemma}

\begin{proof} By the definition of $Q_p,$ w.l.o.g. it suffices to show that $$F(x):=\left(x_1^{\frac1p}+x_2^{\frac1p}\right)^{p}$$ is concave on $\R^N_+$, where $x=(x_1,x_2,\cdots,x_N).$ Direct computation shows that
  \begin{eqnarray*}
    &&\frac{\partial^2F}{\partial x_1^2}=-\frac{p-1}{p}\left(x_1^{\frac1p}+x_2^{\frac1p}\right)^{p-2}x_1^{\frac1p-2}x_2^{\frac1p}\leq 0,\\
    &&\frac{\partial^2F}{\partial x_2^2}=-\frac{p-1}{p}\left(x_1^{\frac1p}+x_2^{\frac1p}\right)^{p-2} x_1^{\frac1p}x_2^{\frac1p-2}\leq 0,\\
    &&\frac{\partial^2F}{\partial x_1^2}\frac{\partial^2F}{\partial x_2^2}-\left(\frac{\partial^2F}{\partial x_1\partial x_2}\right)^2=0,\\
    &&\frac{\partial^2F}{\partial x_i\partial x_j}=0,\quad\ i\neq1,2\ \mathrm{or}\ j\neq 1,2.
  \end{eqnarray*} This yields the concavity of $F.$
\end{proof}
The following lemma is well-known.
\begin{lemma}\label{lem:maximal values}
  Let $\Omega\subset \R^n$ be a precompact domain and $H:\Omega\to \R$ be a $C^1$ concave function. Then any critical point of $H$ in $\Omega$ attains the maximum of $H$ over $\Omega.$ That is, if $y\in \Omega$ such that the differential of $H$ vanishes at $y,$ then $H(y)=\max\limits_{x\in \Omega}H(x).$
\end{lemma}

\begin{thm}\label{pro-for-max-eigen}
Let $G=(V,E,\nu,\mu)$ be a weighted graph and $\Omega$ be a finite  connected bipartite subgraph of $G$. Assume $f$ is an eigenfunction of Dirichlet $p$-Laplacian on $\Omega$. If $f$ satisfies $f(x)f(y)<0$ for $x\sim y$ and $x,y\in \Omega$, then $f$ is a maximum eigenfunction.
\end{thm}
\begin{proof}
Without loss of generality, we assume $\|f\|_{p,\Omega}=1$ and  $f>0$ on $V_1$, then $\bar{f}\in S_1'$. Since $f$ is the eigenfunction pertaining to eigenvalue $\lambda$ of Dirichlet $p$-Laplacian, then $\bar{f}$ is the critical point of $E_p$ on $S_1^p$ pertaining to the critical value $\lambda$. By Lemma \ref{Equi-12} and Lemma
\ref{Equivalence}, we obtain $g=h^p$ with $h=\bar{S}(\bar{f})$ is the
critical point of $Q_p$ with the Lagrange constant $\lambda$.
By Lemma \ref{lem:concavity}, $Q_p$ is concave in $\R_+^N.$ By the restriction, $Q_p$ is concave on the affine subset $P_1$ on which the variational problem \eqref{e:reformulation} defined. Then Lemma \ref{lem:maximal values} yields that $Q_p$ attains the maximum at $g=(\bar{S}(\bar{f}))^p\in P_1.$
Since
$E_p(\bar{f})=Q_p\left((\bar{S}(\bar{f}))^p\right)$, and $T$ and $\bar{S}$ are invertible maps, then
\[
\max\limits_{f\in S_1'}E_p(\bar{f})
=\max\limits_{f\in S_1'}Q_p\left((\bar{S}(\bar{f}))^p\right)
=\max\limits_{g\in P_1}Q_p(g).
\]
Hence, $f$ is a maximum eigenfunction of $E_p$.
\end{proof}

\begin{lemma}\label{maximum-unique}
Let $G=(V,E,\nu,\mu)$ be a weighted graph and $\Omega$ be a finite  connected bipartite subgraph of $G$. Then
the maximum eigenfunction is unique (up to multiplication with constants).
\end{lemma}

\begin{proof}
It suffices to prove that for any two maximum eigenfunctions $u_m$ and $u_m'$, there is a constant $c\in \R,$ such that $u_m=cu_m'.$

Let $h_1=S(u_m)$ and $h_2=S(u_m').$ It suffices to show that $h_1=ch_2.$ By Proposition \ref{opp-sign}, w.l.o.g., we may assume that $h_1,h_2>0$ on $\Omega$ and $\|h_1\|_p=\|h_2\|_p=1.$ Set a new function $h:=(h_1^p+h_2^p)^{1/p}$ and $u=S^{-1}(h)$. Then  $\|h\|_{p,\Omega}^p=\|u\|_{p,\Omega}^p=2$, and
\begin{align}\begin{split}\label{maximu-bdd-lower}
  2\lambda_{m,p}=\|u\|_{p,\Omega}^p\lambda_{m,p}
  &\geq  \frac{1}{2}\sum\limits_{x,y\in V}|\bar{u}(y)-\bar{u}(x)|^p\mu_{xy} \\
 &=  \frac{1}{2}\sum\limits_{x,y\in V}|\bar{h}(y)+\bar{h}(x)|^p\mu_{xy}. \end{split}\end{align}
For any $x,y\in V$ and $x\sim y$, we claim that
\[
|\bar{h}(y)+\bar{h}(x)|^p\geq |\bar{h}_1(y)+\bar{h}_1(x)|^p+|\bar{h}_2(y)+\bar{h}_2(x)|^p\]
and the equality holds if and only if $\frac{h_1(x)}{h_2(x)}=\frac{h_1(y)}{h_2(y)}.$
Note that if $x\in\delta\Omega$ or $y\in\delta\Omega$, then the equality holds. For $x,y\in\Omega$ and $x\sim y$, we define two vectors in the $\ell^p$ space $(\R^2,|\cdot|_p),$ $U=(h_1(x),h_2(x))$ and $V=(h_1(y),h_2(y)).$
Hence \begin{eqnarray*}|h(y)+h(x)|^p&=&(|U|_p+|V|_p)^p\\
&\geq& |U+V|_p^p=|h_1(y)+h_1(x)|^p+|h_2(y)+h_2(x)|^p.
\end{eqnarray*} The equality holds if and only if $U=cV$ with $c>0$. This proves the claim.
Hence, \begin{align}\begin{split}\label{maxi-bdd-up}
&\frac{1}{2}\sum\limits_{x,y\in V}|\bar{h}(y)+\bar{h}(x)|^p\mu_{xy}\\
&\geq\frac{1}{2}\sum\limits_{x,y\in V}|\bar{h}_1(y)+\bar{h}_1(x)|^p\mu_{xy}
  +\frac{1}{2}\sum\limits_{x,y\in V}|\bar{h}_2(y)+\bar{h}_2(x)|^p\mu_{xy}\\ &  = 2\lambda_{m,p}.
\end{split}\end{align}
All the inequalities have to be equalities.
By the claim above and the connectedness of $\Omega$, we have $h_1=c h_2.$ This proves the lemma.
\end{proof}
\begin{proof}[Proof of Theorem \ref{maxi-funct}]
Combining Theorem \ref{pro-for-max-eigen} with Lemma \ref{maximum-unique}, we obtain Theorem \ref{maxi-funct}.
\end{proof}
\section{Cheeger constants on symmetry  graphs}
In this section, we use the uniqueness property of first Dirichlet eigenfunctions to simplify the calculation of the Cheeger constant of a symmetric graph. Let $\Omega$ be a subgraph of a weighted graph $G=(V,E,\nu,\mu)$. Recall that the Cheeger constant $h_{\mu,\nu}(\Omega)$ of $\Omega$ is defined as in (\ref{eq:Cheegerconstant}).
For a finite subset $\Omega,$ the infimum can be attained by some subsets $U$ which we call the Cheeger cuts.
The following example shows that the Cheeger cuts are usually not unique, see also \cite{Chang-Shao-Zhang17a}.
\begin{figure}[htbp]
 \begin{center}
   \begin{tikzpicture}
    \node at (0,0){\includegraphics[width=0.6\linewidth]{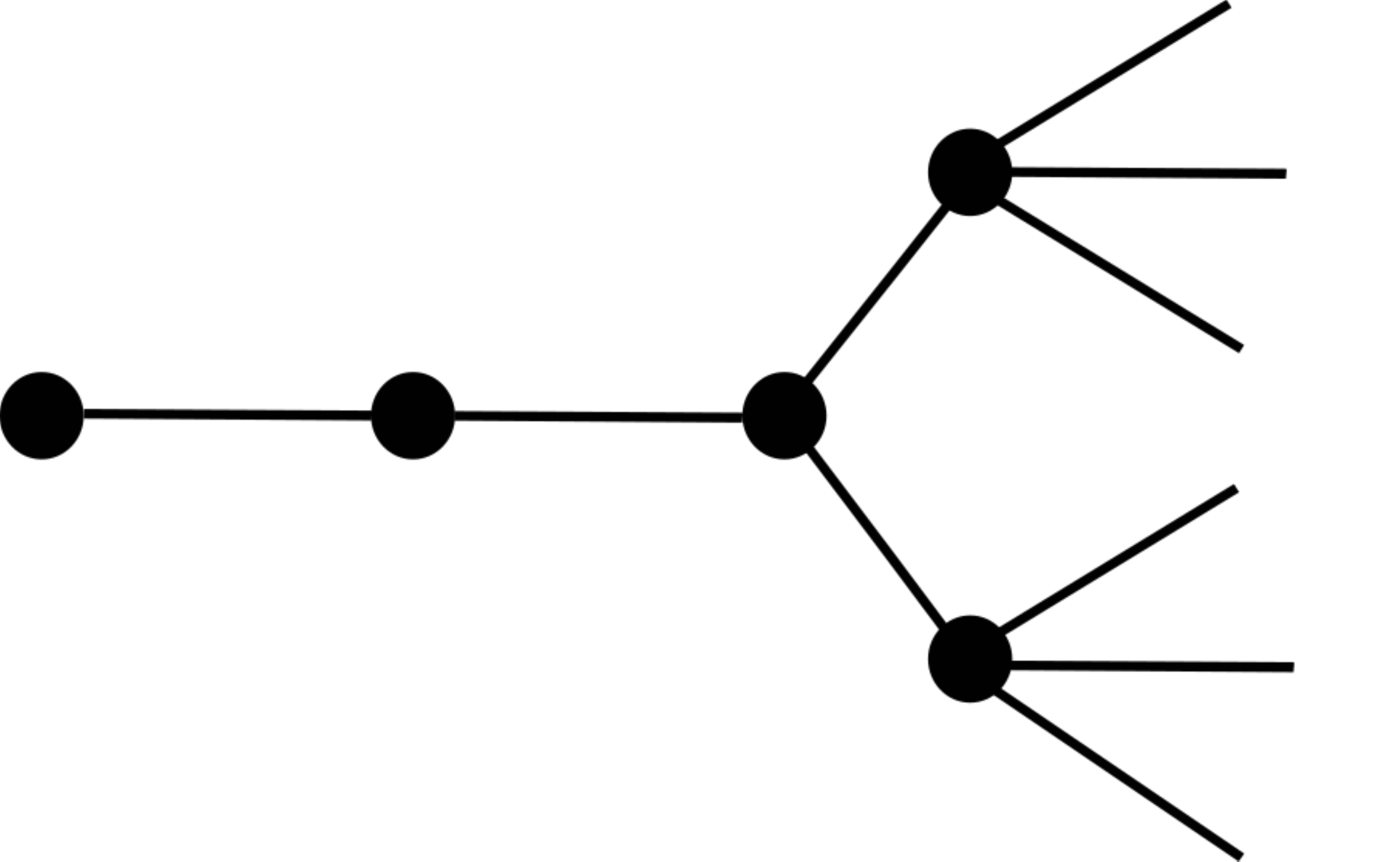}};
    \node at (-3.6, -.5){\Large $v_1$};
    \node at (-1.6, -.5){\Large $v_2$};
    \node at (0.3, -.5){\Large $v_3$};
    \node at (1.5, 0.8){\Large $v_4$};
    \node at (1.5,   -1.8){\Large $v_5$};
    \end{tikzpicture}
  \caption{\small }
\label{fig2}
 \end{center}
\end{figure}
\begin{example}\label{cut-non-unique}
Let $G=(V,E,\nu,\mu)$ be a weighted graph with $\mu_{v_iv_j}=1$ for $v_i\sim v_j$ and $v_i,v_j\in V$,  and $\nu_{v_i}=\sum\limits_{v_j\in V, v_i\sim v_j}\mu_{v_iv_j}$.
For a subgraph of $G$, $\Omega=\{v_1,v_2,v_3,v_4,v_5\}$, as shown in Figure \ref{fig2}, 
\[
h_{\mu,\nu}(\Omega)=\frac{|\partial K_1|_\mu}{|K_1|_\nu}=\frac{|\partial K_2|_\mu}{|K_2|_\nu},
\]
where $K_1=\{v_1,v_2,v_3\}$ and $K_2=\{v_1, v_2\}$.
\end{example}

\subsection{Cheeger constants of symmetric subgraph}
Let $\Omega$ be a possibly infinite subgraph of $G$ and $\Omega^c$ be the complement of
$\Omega$. For the Dirichlet problem on $\Omega,$
it will be convenient to consider a simplified model
$\Omega\cup \{\underline{o}\},$ in which we identify
$\Omega^c$ as a single point $\{\underline{o}\}.$
We embed $\Omega$ in $\Omega\cup \{\underline{o}\}$ with the same weights $\mu$ and $\nu.$
For any $x\in \Omega$ which connects to a vertex in $\Omega^c,$
i.e. there exists $y\in \Omega^c$ such that $x\sim y,$
we add a new edge $\{x,\underline{o}\}$ and put the weight
\[\mu_{x\underline{o}}=\sum_{y\sim x, y\in \Omega^c}\mu_{xy}.\]
The vertex measure $\nu$ at ${\underline{o}}$ is irrelevant to the Dirichlet problem,
and can be chosen arbitrarily, say $\nu_{\underline{o}}=1$.
It is ready to see that the Dirichlet $p$-Laplacian on $\Omega$ in $G$ is equivalent to the Dirichlet $p$-Laplacian on $\Omega$ in $\Omega\cup \{\underline{o}\}.$

A bijective map $g:\Omega\to \Omega$ is called $D$-automorphism(standing for D irichlet-automorphism)
of $\Omega$ if
\[\nu_{g(x)}=\nu_x,\ \mu_{g(x)g(y)}=\mu_{xy}, \ \mu_{g(x)\underline{o}}=\mu_{x\underline{o}},\ \forall
x,y\in \Omega.\]
The set of $D$-automorphism of $\Omega,$ denoted by $\Aut(\Omega)$,
forms a subgroup of the permutation group $S_N$ where $N$ is the cardinality of $\Omega.$
Given any subgroup $\Gamma$ of the $D$-automorphism group $\Aut(\Omega),$
we denote by $[x]$ the $\Gamma$-orbit of $x\in\Omega$ under the action of $\Gamma$.
When the group $\Gamma$ is clear in the context, we simply call a $\Gamma$-orbit an orbit. We say that a subgraph $\Omega$ is ``symmetric" if there is a subgroup $\Gamma$ of $\Aut(\Omega)$ acting on $\Omega$ finitely, i.e. each orbit for the action of the group $\Gamma$ consists of finitely many vertices. We denote the set of $\Gamma$-orbits in $\Omega$ by $\Omega/\Gamma$.
\begin{defi}\label{quotient-graph-def}
We call a graph $G(\Omega/\Gamma)=(V^\Gamma, E^\Gamma,\nu^\Gamma,\mu^\Gamma)$ with $V^\Gamma=\{[x]: x\in \Omega\}\cup \Omega^c$ a quotient graph on $\Omega$ of $G$ by $\Gamma$ if
\begin{enumerate}
\item
      \begin{align*}\nu_z^\Gamma=\begin{cases}
      \sum\limits_{x_1\in [x]}\nu_{x_1}, \ & z=[x],   x\in\Omega\\
      \nu_z,  & z\in\Omega^c;
      \end{cases}\end{align*}
 \item For $z,w \in V^\Gamma$,  \begin{align*}\mu_{zw}^\Gamma=\begin{cases}
      \sum\limits_{x_1\in [x],y_1\in [y]}\mu_{x_1y_1}, \ \ &z=[x],\ w=[y],\ x,y\in \Omega\\
      \sum\limits_{x_1\in [x]}\mu_{x_1w},  &z=[x],\ w\in\Omega^c, x\in\Omega\\
      \mu_{zw}, & z,w\in \Omega^c;
      \end{cases}\end{align*}
   \item $\{z,w\}\in E^\Gamma$ if and only if $\mu_{zw}>0$.
\end{enumerate}
If $\Omega=V$, then we write $G/\Gamma=G(\Omega/\Gamma)$ for simplicity.
\end{defi}
Note that in our definition, the quotient graph $G(\Omega/\Gamma)$ has no self-loops, although there could be edges between vertices in one orbit in $\Omega$.
For the Dirichlet problem of the quotient graph $G(\Omega/\Gamma)$ on $\Omega/\Gamma$,
we simply set $\Omega/\Gamma\cup\{\underline{o}\}$ with edge weights
$\mu_{[x]\underline{o}}=\sum_{x_1\in [x]}\mu_{x_1\underline{o}}.$

For any function $f$ defined on $\Omega/\Gamma,$
we have a natural lifted function $\hat{f}\in \R^\Omega$, $\hat{f}(x)=f([x])$ for any $x\in \Omega.$
We denote by $E_p^{\Omega/\Gamma}$ the $p$-Dirichlet energy of functions on the quotient graph $G(\Omega/\Gamma)$,
it easy to check that
\begin{equation}\label{e:energy relation}
E_p^{\Omega/\Gamma}(f)=E_p(\hat{f})\ \ \ \quad  \forall\ f\in \R^{\Omega/\Gamma},
\end{equation}
from the choice of weights.

Next we state the main result, which yields that we can treat the orbits as integrality in the computation of Cheeger constants for symmetric graphs.
\begin{thm}\label{cheeger-equ}
Let $\Omega$ be a finite connected subgraph of a weighted graph $G=(V,E,\nu,\mu)$ and $\Gamma$ be a subgroup of
$\Aut(\Omega)$ acts finitely. Then in the quotient graph $G(\Omega/\Gamma)$,
\begin{equation}
\label{e:quotient Cheeger}h_{\mu,\nu}(\Omega)=h_{\mu,\nu}\left(G(\Omega/\Gamma)\right).
\end{equation}
\end{thm}

Instead of combinatorial arguments, we use an analytic approach to prove this result.
The following lemma follows from the fact that
the first eigenvalue of $p$-Laplacian is simple for $p\in(1,\infty)$, i.e. the first eigenfunctions are of dimension one.
\begin{lemma}\label{equivalence}
  Let $\Omega$ be a finite connected subgraph of a weighted graph $G=(V,E,\nu,\mu)$ and $\Gamma$ be a subgroup of $\Aut(\Omega)$ acts finitely.
  Then for any $p\in (1,\infty),$ we have $$\lambda_{1,p}(\Omega)=\lambda_{1,p}(\Omega/\Gamma).$$
\end{lemma}
\begin{proof}
  Let $u_1$ be the first eigenfunction of $\Omega$ in $G.$
  Then by the symmetry, for any $D$-automorphism $T,$
  $u_1\circ T$ is also a first eigenfunction of $\Omega.$
  By Lemma \ref{one-multi}, there is $c>0$ such that $u_1=c u_1\circ T.$
  Note that
  \[
  \|u_1\|_{p}=\left(\sum_{x\in\Omega}|u_1(x)|^p\nu_x\right)^\frac{1}{p}
  =\left(\sum_{x\in\Omega}|u_1\circ T(x)|^p\nu_{T(x)}\right)^\frac{1}{p}
  =\|u_1\circ T\|_{p},
  \]
 then $c=1.$
  Hence,  $u_1=u_1\circ T$ for all $T\in \Gamma$.
  This means that $u_1$ is constant on each orbit $[x],$ $x\in \Omega.$
  This induces a function $\underline{u}_1$ on $\Omega/\Gamma$ and $\underline{u}_1$ has the same sign condition as $u_1$. 
  Obviously either $\hat{\underline{u}}_1=u_1$.
  By our setting of measure $\mu^\Gamma$ and $\nu^\Gamma$ for the quotient graph,  $\underline{u}_1$ is an eigenfunction of $\Omega/\Gamma.$ 
  By Theorem \ref{First-eigen}, either $u_1>0$ on $\Omega$ and $u_1<0$ on $\Omega$. Then either $\underline{u}_1>0$ on $G/\Gamma$ or  $\underline{u}_1<0$ on $G/\Gamma$.
  Theorem \ref{First-eigen} yields that $\underline{u}_1$ is the first eigenfunction of $G(\Omega/\Gamma)$.
  This proves the theorem.
\end{proof}
\begin{proof}[Proof of Theorem \ref{cheeger-equ}]
By the Rayleigh quotient characterization (\ref{Rayleigh quotient}), \[\lim\limits_{p\to 1}\lambda_{1,p}(\Omega)=\lambda_{1,1}(\Omega).\] Combining Lemma \ref{equivalence} with Proposition \ref{c:1-Lap}, we obtain Theorem \ref{cheeger-equ}.
\end{proof}

\subsection{Cheeger constants of infinite graphs}
Let $G$ be an infinite graph and $\{\Omega_i\}_{i=1}^{\infty}$ be an exhaustion of $G$, i.e. $\Omega_i$ are finite subsets, $\Omega_i\subset\Omega_{i+1}$ and $V=\cup_{i=1}^\infty \Omega_i.$

\begin{defi}
We define the bottom of the spectrum by
$$\lambda_{1,p}(G)=\lim_{i\to \infty}\lambda_{1,p}(\Omega_i),$$
and the bottom of the essential spectrum by
$$\lambda_{1,p}^{\ess}(G)=\lim_{i\to\infty}\lambda_{1,p}(V\setminus \Omega_i),$$
where \[\lambda_{1,p}(V\setminus\Omega_i)=\inf\limits_{f\in C_0(V\setminus \Omega_i), f\not\equiv  0  }\frac{\|\nabla \bar{f}\|^p_{p,E}}{\|f\|_{p,V\setminus\Omega_i}^p}\]
and $C_0(V\setminus \Omega_i):=\Big\{f\in \R^{V\setminus \Omega_i}\Big|\#\{x\in V\setminus \Omega_i|f(x)\neq 0\}<\infty\Big\}$.
\end{defi}

The Rayleigh quotient characterization implies that $\lambda_{1,p}(\Omega_2)\leq \lambda_{1,p}(\Omega_1)$ for $\Omega_1\subset\Omega_2,$
so $\lambda_{1,p}(G)$ and $\lambda_{1,p}^{\ess}(G)$ are independent of the choice of the exhaustion.
Similarly, we can define the Cheeger constant of an infinite graph and the Cheeger constant at infinity of an infinite graph as follows.

\begin{defi}\label{def:Cheegerinf}The Cheeger constant of an infinite graph is defined as
$$h(G)=\lim_{i\to \infty}h(\Omega_i),$$
and the Cheeger constant at infinity is defined as
$$h_{\infty}(G)=\lim_{i\to\infty}\lim_{j\to\infty}\lambda_{1,p}(\Omega_j\setminus \Omega_i).$$
\end{defi}
For normalized Laplacians, the Cheeger estimates for infinite graphs are well-known. As a consequence, $\lambda_{1,2}(G)=0$ if and only if $h(G)=0,$ see e.g. \cite{Fujiwara96}.
\begin{lemma}[\cite{Keller-Mugnolo16}] For a weighted graph $G=(V,E,\nu,\mu)$ with normalized $p$-Laplacian, $1< p<\infty,$ then
  $$2^{p-1}\left(\frac{h(G)}{p}\right)^{p}\leq \lambda_{1,p}(G)\leq h(G),$$
  $$2^{p-1}\left(\frac{h_{\infty}(G)}{p}\right)^{p} \leq \lambda_{1,p}^{\ess}(G) \leq h_{\infty}(G).$$
\end{lemma}

\begin{proof}[Proof of Theorem \ref{quotient-graph-Chee}]
On one hand, we claim that there exists an exhaustion $\{\Omega_i\}_{i=1}^\infty$ such that each $\Omega_i$ is connected and consists of vertices in a collection of $\Gamma$-orbits.
For $x_0\in V,$ choose a large constant $R>0$ such that $\{x\in V: x\in [x_0]\}\subset B_R(x_0).$ Note that any ball in a graph is connected since any vertex in the ball can be connected to the center of the ball by a path of finite length. We denote by $B_r([x_0])$ the ball of radius $r$ centered at $[x_0]$ in the quotient graph $G/\Gamma.$ Then for any $i\geq R,$ set $$\Omega_i:=\{x\in V: [x]\in B_i([x_0])\}.$$ It is easy to see that $\{\Omega_i\}_{i=1}^\infty$ is an exhaustion of $G$ and consists of vertices of some $\Gamma$-orbits. Note that for any $[x]\sim [y]$ in $G/\Gamma$ and any vertex $x'\in [x],$ there is a vertex $y'\in [y]$ such that $x'\sim y'$ in $G.$ For any $[x]\in B_i([x_0]),$ there is a path connecting $[x]$ and $[x_0],$ denoted by $$[x]\sim [x_N]\sim [x_{N-1}]\sim \cdots \sim [x_0], $$ where $[x_k]\in B_i([x_0]), 1\leq k \leq N.$ Hence for any $x'\in [x],$ there exists $x_j'\in [x_j]$ for $0\leq j\leq N$ such that $$x'\sim x_N'\sim x_{N-1}'\sim \cdots \sim x_0',$$ and $x_k'\in \Omega_i$, $0\leq k\leq N.$ Note that $x_0'\in [x_0]\subset B_R(x_0),$ there exists a path connecting $x_0'$ and $x_0$ in $B_R(x_0).$ Hence by $i\geq R,$ there is a path connecting $x'$ and $x_0$ in $\Omega_i.$ This proves the claim.

On the other hand, $\Gamma$ acts finitely on each $\Omega_i$ and $\Omega_j\setminus \Omega_i$ for any $i<j.$
  The theorem follows from Theorem \ref{cheeger-equ} and Definition \ref{def:Cheegerinf}.
\end{proof}

\section{Cheeger constants on spherically symmetric graphs}
For the combinatorial distance $d$, we denote the balls, spheres and annuli as follows.
For $x\in V,$ $0<r<R$,
\[
B_r(x):=\{y\in V: d(y,x)\leq r\},\ \ S_r(x):=\{y\in V: d(y,x)=r\},
\]
and
\[
A_{r,R}(x):=B_R(x)\setminus B_r(x)=\{y\in V: r<d(y,x)\leq R\}.
\]
Recall that a graph $G$ is \textbf{spherically symmetric} centered at $x_0 \in V $ if for any $x,y\in S_r(x_0)$ and $r\in \mathbb N\cup \{0\}$, there exits an automorphism of $G$ which leaves $x_0$ invariant and maps $x$ to $y$. Let $\Gamma$ be a subgroup of the automorphism group on the spherically symmetric graph $G$ centered at $x_0\in V$ such that $\Gamma x_0=x_0$ and $\Gamma$-orbits are exactly $\{S_r(x_0)\}_{r=0}^\infty.$ We call $\Gamma$ an associated automorphism subgroup. Hence, the set of quotient graph $G/\Gamma$ is given by
$\{S_r(x_0)\}_{r=0}^\infty.$


Next we introduce a ``one-dimensional" model graph for the quotient graph.
\begin{defi}\label{def:one-dimen}Let $\Z_{\geq 0}:=\N\cup\{0\}$ be nonnegative integers.
Set a graph based on $\Z_{\geq 0,}$ $L=(\underline{V},\underline{E}),$ where $\underline{V}=\Z_{\geq 0}$
and $\underline{E}=\{\{i,i+1\}: i\in \Z_{\geq 0}\}.$
The weight on $\underline{V}$ is given by
\[
\nu: i\mapsto\nu_i,
\ \ \textrm{for}\  i\in\underline{V},\ \nu_i\in \R_+,
\]
and the weight on $\underline{E}$ is given by
\[\mu:\{i,i+1\}\mapsto \mu_{i,i+1}, \ \ \textrm{for}\  \{i,i+1\}\in\underline{E}, \ \mu_{i,i+1}\in\R_+.\]\end{defi}
For convenience, we denote $\mu_i:=\mu_{i,i+1}$ for $i\in \Z_{\geq 0},$
and
\begin{equation}\label{ball-anu-model}
\underline{B}_r:=B_r(0),\ \ \underline{A}_{r,R}:=B_R(0)\setminus B_r(0).
\end{equation}
We call it a linear graph.
Furthermore, for any subset $U\subset \underline{V}$ and $F\subset \underline{E},$ we denote
$|U|_{\nu}:=\sum_{x\in U}\nu_x$ and $|F|_{\mu}:=\sum_{e\in E}\mu_e.$
We simply write $|\cdot|$ if the measure is evident.
For the graph $L,$ by setting $\mu_{-1}=0$ and $\underline{A}_{-1,r}=\underline{B}_r,$ we have
$$
|\partial \underline{A}_{k,r}|=\mu_{k}+\mu_r, \qquad
|\underline{A}_{k,r}|=\sum_{i=k+1}^r \nu_i,\ \ \forall\ 0\leq k+1\leq r,\ k,r\in\Z.
$$

For the model graphs, we have the following estimate.
\begin{satz}\label{p:line graph}
Let $L$ be a graph defined in Definition \ref{def:one-dimen} and $\underline{\Omega}$ be a connected subgraph of $L$.  Then
  \begin{enumerate}[(a)]
  \item $$h(\underline{\Omega})=\inf_{\underline{B}_r\subset \underline{\Omega}}\frac{|\partial \underline{B}_r|}{|\underline{B}_r|}\ \ \textrm{if} \ \ 0\in \underline{\Omega},$$
  \item $$h(\underline{\Omega})=\inf_{\underline{A}_{k,r}\subset \underline{\Omega}}\frac{|\partial \underline{A}_{k,r}|}{|\underline{A}_{k,r}|} \ \ \textrm{if}\ \ 0\notin \underline{\Omega},$$
  \end{enumerate}
  where the ball $\underline{B}_r$ and annulus $\underline{A}_{k,r}$ of $L$ are defined as (\ref{ball-anu-model}).
\end{satz}
\begin{proof}
  $(a)$ For any $U\subset \underline{\Omega},$ denote $R(U):=\max_{i\in U} i.$
  Then it is easy to see that $B_{R(U)}\subset \underline{\Omega}$ and
  $$\frac{|\partial \underline{B}_{R(U)}|}{|\underline{B}_{R(U)}|}\leq \frac{|\partial U|}{|U|}.$$
  Hence the Cheeger constant can be calculated using only balls $\underline{B}_R,$ $R\geq 0.$ This proves the result.

  $(b)$ For any $U\subset \underline{\Omega},$
  denote $R(U):=\max\limits_{i\in U} i$ and $r(U):=\min\limits_{i\in U} (i+1).$
  This yields that $\underline{A}_{r(U),R(U)}\subset \Omega$ and
  $$\frac{|\partial \underline{A}_{r(U),R(U)}|}{|\underline{A}_{r(U),R(U)}|}\leq \frac{|\partial U|}{|U|}.$$
This proves the proposition.
\end{proof}
 For a spherically symmetric graph $G$ with the associated automorphism subgroup $\Gamma$, let $G/\Gamma$ be the quotient graph. Then it is easy to see that $L:=G/\Gamma$ is a linear graph defined in
 Definition \ref{def:one-dimen}.
\begin{lemma}\label{radial-symmetry}Let $G$ be a spherically symmetric graph centered at $x_0\in V$ with the associated subgroup $\Gamma$.
Then
\begin{equation*}\label{e:radial symmetry eq3}
h_{\infty}(L)=\liminf_{r\to \infty}\inf_{R\geq r+1}\frac{|\partial \underline{A}_{r,R}|}{|\underline{A}_{r,R}|},
\end{equation*}
where the ball $\underline{B}_r$ and annulus $\underline{A}_{k,r}$ of $L:=G/\Gamma$ are defined in (\ref{ball-anu-model}).
\end{lemma}

\begin{proof}
Applying Proposition \ref{p:line graph}, together with $\{\underline{B}_l\}_{l=0}^\infty$ is an exhaustion of $L$, we have
  \begin{align}\begin{split}\label{e:eq h1}
  h_{\infty}(L)=\lim_{l\to\infty}h(L\setminus \underline{B}_l)
  &=\lim_{l\to\infty}\inf_{\substack{s,t\geq l\\ t\geq s+1}}\frac{|\partial \underline{A}_{s,t}|}{|\underline{A}_{s,t}|}.
  \end{split}\end{align}
  For convenience, set $J:=\liminf\limits_{r\to \infty}\inf\limits_{R\geq r+1}\frac{|\partial \underline{A}_{r,R}|}{|\underline{A}_{r,R}|}.$
  Clearly, $h_{\infty}(G)\leq J.$ It suffices to prove $h_{\infty}(L)\geq J$.
  By \eqref{e:eq h1}, there exist sequences $\{s_l\}_{l=1}^\infty$ and $\{t_l\}_{l=1}^\infty$ such that $s_l,t_l\geq l,$ $t_l\geq s_l+1,$ and
  $$\lim_{l\to \infty}\frac{|\partial \underline{A}_{s_l,t_l}|}{|\underline{A}_{s_l,t_l}|}= h_{\infty}(L).$$
  Noting that
  $$\inf_{R\geq s_l+1}\frac{|\partial \underline{A}_{s_l,R}|}{|\underline{A}_{s_l,R}|}\leq\frac{|\partial \underline{A}_{s_l,t_l}|}{|\underline{A}_{s_l,t_l}|},\quad \forall\ l\in \N,$$
  we have
  \begin{eqnarray*}
    J&=&\liminf_{r\to \infty}\inf_{R\geq r+1}\frac{|\partial \underline{A}_{r,R}|}{| \underline{A}_{r,R}|}
    \leq \liminf_{l\to\infty}\inf_{R\geq s_l+1}\frac{|\partial \underline{A}_{s_l,R}|}{| \underline{A}_{s_l,R}|}\\
    &\leq&\liminf_{l\to\infty}\frac{|\partial \underline{A}_{s_l,t_l}|}{| \underline{A}_{s_l,t_l}|}=h_{\infty}(L).
  \end{eqnarray*}
We prove the lemma.
\end{proof}
For further computation of Cheeger constants, we state a useful lemma as follows.
\begin{lemma}\label{l:calculus lemma1}
Let $G$ be a spherically symmetric graph centered at $x_0\in V$ with the associated subgroup $\Gamma$. If the graph $G$ has infinite volume, i.e. $\lim\limits_{r\to \infty}|B_r(x_0)|= \infty,$ then
\[
h_\infty(L)=\liminf_{r\to\infty}\frac{|\partial \underline{B}_r|}{|\underline{B}_r|}.
\]
where the ball $\underline{B}_r$ of $L:=G/\Gamma$ are defined in (\ref{ball-anu-model}).
\end{lemma}
\begin{proof}
By $L=G/\Gamma$, $L$ is a quotient graph of $G$. From Definition \ref{quotient-graph-def}, we have $|\underline{B}_r|=|B_r|$, then $\lim\limits_{r\to\infty}|\underline{B}_r|=\infty$.
For any $r\geq k+1,$
$$\frac{|\partial \underline{A}_{k,r}|}{|\underline{A}_{k,r}|}=\frac{|\partial \underline{B}_r|+|\partial \underline{B}_k|}{|\underline{B}_r|-|\underline{B}_k|}.$$
Hence for fixed $k,$
\begin{align}\begin{split}\label{annu-bdd-ball}
\inf_{r\geq k+1}\frac{|\partial \underline{A}_{k,r}|}{|\underline{A}_{k,r}|}
\leq \liminf_{r\to \infty}\frac{|\partial \underline{B}_r|+|\partial \underline{B}_k|}{|\underline{B}_r|-|\underline{B}_k|}
=\liminf_{r\to\infty}\frac{|\partial \underline{B}_r|}{|\underline{B}_r|},
\end{split}\end{align}
where we used$\lim\limits_{r\to \infty}|\underline{B}_r|= \infty$. By passing to the limit $k\to\infty,$
we prove $$\lim_{k\to\infty}\inf_{r\geq k+1}\frac{|\partial{\underline{A}_{k,r}}|}{|\underline{A}_{k,r}|}\leq\liminf_{r\to\infty}\frac{|\partial \underline{B}_r|}{|\underline{B}_r|}.$$

On the other hand, note that for any $r\geq k+1,$
$\frac{|\partial \underline{A}_{k,r}|}{|\underline{A}_{k,r}|}\geq \frac{|\partial \underline{B}_r|}{|\underline{B}_r|}.$
Hence,
\begin{align}\begin{split}\label{ball-bdd-annu}
\lim_{k\to\infty}\inf_{r\geq k+1}\frac{|\partial \underline{A}_{k,r}|}{|\underline{A}_{k,r}|}
\geq\lim_{k\to\infty}\inf_{r\geq k+1} \frac{|\partial \underline{B}_r|}{|\underline{B}_r|}
=\liminf_{r\to\infty} \frac{|\partial \underline{B}_r|}{|\underline{B}_r|}.
\end{split}\end{align}
Combining (\ref{annu-bdd-ball}) and (\ref{ball-bdd-annu}), we obtain
\[\lim_{k\to\infty}\inf_{r\geq k+1}\frac{|\partial{\underline{A}_{k,r}}|}{|\underline{A}_{k,r}|}=\liminf_{r\to\infty}\frac{|\partial \underline{B}_r|}{|\underline{B}_r|}.\]
Then the lemma follows from Lemma \ref{radial-symmetry}.
\end{proof}

\begin{proof}[Proof of Theorem \ref{t:spherically symmetric}]
  By Theorem \ref{quotient-graph-Chee} and $L=G/\Gamma$, we have
 $$h(G)=h(G/\Gamma)=h(L),\ \ \quad h_{\infty}(G)=h_{\infty}(G/\Gamma)=h_{\infty}(L).$$
 It suffices to calculate the Cheeger constants for the ``one-dimensional" graph $L=G/\Gamma$.
  By Proposition \ref{p:line graph} and $0\in L$,  we have
  \begin{align}\begin{split}\label{e:Cheeger radial}
  h(G)=h(L)=\inf\limits_{\underline{B}_r\subset L}\frac{|\partial\underline{B}_r |}{|\underline{B}_r|}
  =\inf\limits_{r\geq 0}\frac{|\partial\underline{B}_r |}{|\underline{B}_r|}.
  \end{split}\end{align}
  Combining Lemma \ref{radial-symmetry} with Lemma \ref{l:calculus lemma1}, together with (\ref{e:Cheeger radial}), we obtain Theorem \ref{t:spherically symmetric}.
\end{proof}

By (\ref{e:reduction to balls}) and the well-known Stolz-Ces\`{a}ro Theorem, see e.g. \cite[Theorem 1.22]{Lindenstrauss-Tzafriri77}, for $h_\infty(G)$, it is sufficient to compute the following quantity in next lemma.
This simplifies the computation in many cases.
\begin{lemma}[\cite{Lindenstrauss-Tzafriri77}]\label{l:calculus lemma2}
Let $L$ be a linear graph with infinite $\nu$-total measure.
If $$\lim_{r\to\infty}\frac{|\partial \underline{B}_r|_\mu-|\partial \underline{B}_{r-1}|_\mu}{|\underline{S}_r|_\nu}=A\in \R\cup \{\infty\},$$
then  $$\lim_{r\to\infty}\frac{|\partial \underline{B}_r|_\mu}{|\underline{B}_r|_\nu}=A.$$
\end{lemma}

Applying these above lemmas, we can calculate various Cheeger constants for spherically symmetric graphs, see Appendix.
\appendix

\section{Examples}
In this Appendix, we consider spherically symmetric graphs and calculate their Cheeger constants.
We always assume that the edge weight $\mu$ is trivial, i.e. $\mu_{xy}=1,$ for any $x\sim y.$
 For the weighted graph $G=(V,E,\nu,\mu),$ we set
 $$\deg(x):=\sum\limits_{y\in V}\mu_{xy},\quad
\textrm{Deg}(x):=\frac{\deg(x)}{\nu(x)}, \ \forall x\in V.$$ We introduce some definitions as follows.

\begin{defi}\label{n:notation} Let $G=(V,E,\nu,\mu)$ be the weighted graph.
\begin{enumerate}
\item If $\nu_x=1,\ \forall x\in V,$ then we call it a weighted graph with physical Laplacain and denote by $h:=h_{\mu,1}$ the physical Cheeger constant.
\item If $\nu_x=1,\ \forall x\in V,$ then we call $(V,E,\nu,\mu')$ a weighted graph with modified physical Laplacain and denote by $h:=h_{\mu',1}$ the modified Cheeger constant, where $\mu'_{xy}:=\mu_{xy}\rho'(x,y)$ with \begin{equation*}\label{e:intrinsic 1}
\rho'(x,y)=\frac{1}{\sqrt{\textrm{Deg}(x)}}\wedge\frac{1}{\sqrt{\textrm{Deg}(y)}}, \qquad \forall\ x\sim y.
\end{equation*}

\item If $\nu_x=\deg(x),\ \forall x\in V,$ then we call it a weighted graph with normalized Laplacian and denote by $h_N:=h_{\mu,\deg}$ the normalized Cheeger constant.
\item The Cheeger constants at infinity are defined similarly as in (\ref{eq:Cheegerconstant}), denoted by $h_{\infty},$ $h_{\infty,M}$, and $h_{\infty,N},$ respectively.
\end{enumerate}
\end{defi}

\subsection{Fujiwara's spherically symmetric trees}
Let $T$ be an infinite spherically symmetric tree with branching numbers $\{m_i\}_{i=0}^\infty$, i.e. for any $x\in \partial B_i(x_0)$, there are $m_i+1$ neighbors for $i\geq 1$ and $m_0$ neighbors for $i=0$.
Fujiwara \cite{Fujiwara96} proved that the essential spectrum of $T$ is $\{1\}$ if $m_i\to\infty (i\to\infty)$
which is called a rapidly branching tree.
By calculating the Cheeger constants at infinity for the radial symmetric trees,
we can show that this is also a necessary condition, see Corollary \ref{bran-num-ess-spectr}.
For the sake of convenience,
we introduce the following convention that $\sum_{i=0}^{-1}(\cdot)=0$ and $\prod_{i=0}^{-1}(\cdot)=1.$
We list the measures for various Laplacians in the model graph as follows:
\begin{enumerate}
  \item For the physical Laplacian, $\mu_r=\prod_{i=0}^rm_i$ and $\nu_r=\prod_{i=0}^{r-1}m_i$ for any $r\in \Z_{\geq 0}.$
  \item For the modified physical Laplacian, $\nu_r=\prod_{i=0}^{r-1}m_i$ for any $r\in \Z_{\geq 0}.$
      Since \begin{align*}\textrm{Deg}(r)=
      \begin{cases}m_r+1, \ \ &r\geq 1\\  m_0, & r=0,\end{cases}\end{align*}
       \begin{align*}\rho'(r,r+1)=
      \begin{cases}\frac{1}{\sqrt{m_{r+1}+1}}\wedge \frac{1}{\sqrt{m_r+1}} ,\ \ &r\geq 1\\
      \frac{1}{\sqrt{m_1+1}}\wedge\frac{1}{\sqrt{m_0}},
      & r=0. \end{cases}\end{align*}
   Hence, \begin{align*}
   \mu_r'=\mu_r\rho'(r,r+1)=
   \begin{cases}
      \prod_{i=0}^r m_i\left(\frac{1}{\sqrt{m_{r+1}+1}}   \wedge \frac{1}{\sqrt{m_r+1}}\right),\ \ & r\geq 1\\
   m_0\left(\frac{1}{\sqrt{m_1+1}}\wedge \frac{1}{\sqrt{m_0}}\right)
    & r=0.
   \end{cases}
   \end{align*}

\item For the normalized Laplacian, $\mu_r=\prod_{i=0}^rm_i$ and $\nu_r=\prod_{i=0}^{r-1}m_i+\prod_{i=0}^{r}m_i$ for any $r\in \Z_{\geq 0}.$
\end{enumerate}

\begin{example}\label{exam1}
  Let $T$ be the spherically symmetric tree of branching numbers $\{m_i\}_{i=0}^\infty.$ Then Cheeger constants and Cheeger constants at infinity can be calculated as follows:
  \begin{enumerate}
    \item \begin{equation*}h(G)=\inf_{r\geq0}\frac{\prod_{i=0}^rm_i}{\sum_{k=0}^r\prod_{i=0}^{k-1}m_i},\quad h_{\infty}(G)=\liminf_{r\to\infty}\frac{\prod_{i=0}^rm_i}{\sum_{k=0}^r\prod_{i=0}^{k-1}m_i}. \end{equation*}

   \item\begin{align*}
   h_M(G)=\min\Big\{
          &\inf_{r\geq1}
          \frac{\prod_{i=0}^rm_i}{\sum_{k=0}^r\prod_{i=0}^{k-1}m_i}
          \Big(\frac{1}{\sqrt{m_{r+1}+1}}\wedge \frac{1}{\sqrt{m_r+1}}\Big),\\
          & \frac{m_0}{\sqrt{m_1+1}}\wedge\sqrt{m_0}\Big\}.
              \end{align*}
        \begin{equation*}h_{\infty,M}(G)=\liminf_{r\to\infty}\frac{\prod_{i=0}^rm_i}{\sum_{k=0}^r\prod_{i=0}^{k-1}m_i}
          \Big(\frac{1}{\sqrt{m_{r+1}+1}}\wedge \frac{1}{\sqrt{m_r+1}}\Big).\end{equation*}
    \item \begin{equation*}h_N(G)=\inf_{r\geq 0}\frac{\prod_{i=0}^rm_i}{2\sum_{k=0}^{r-1}\prod_{i=0}^km_i+\prod_{i=0}^rm_i},\end{equation*}
    \begin{equation}\label{e:example1 eq1}h_{\infty,N}(G)=\lim\limits_{r\to\infty}\frac{\prod_{i=0}^rm_i}{2\sum_{k=0}^{r-1}\prod_{i=0}^km_i+\prod_{i=0}^rm_i}.\end{equation}
  \end{enumerate}
\end{example}

\begin{coro}\label{bran-num-ess-spectr}
  For the spherically symmetric tree $T$ with branching numbers $\{m_i\}_{i=0}^\infty,$  $$h_{\infty,N}(G)=1\quad \mathrm{if}\ \mathrm{and}\ \mathrm{only}\ \mathrm{if}\ \quad \lim_{i\to\infty} m_i=\infty.$$
\end{coro}
\begin{proof}
  Assume that $\lim\limits_{i\to\infty} m_i=\infty.$ By \eqref{e:example1 eq1}, we know that $h_{\infty,N}=1$ is equivalent to
  \begin{equation}\label{e:example1 eq2}\limsup_{r\to\infty}\frac{\sum_{k=0}^{r-1}\prod_{i=0}^k m_i}{\prod_{i=0}^rm_i}=0,\quad \mathrm{i.e.}\quad\lim_{r\to\infty}\frac{\sum_{k=0}^{r-1}\prod_{i=0}^k m_i}{\prod_{i=0}^rm_i}=0.\end{equation} By Stolz-Ces\`{a}ro theorem  \cite[Theorem 1.22]{Lindenstrauss-Tzafriri77}, this follows from $$\lim_{r\to\infty}\frac{\prod_{0}^{r-1}m_i}{\prod_{0}^{r}m_i-\prod_{0}^{r-1}m_i}=\lim_{r\to\infty}\frac{1}{m_r-1}=0.$$

  For the other direction, $h_{\infty,N}=1$ implies \eqref{e:example1 eq2}. It is easy to see that
  $$\frac{1}{m_r}\leq \frac{\sum_{k=0}^{r-1}\prod_{i=0}^k m_i}{\prod_{i=0}^rm_i}.$$ Taking the limit $r\to\infty,$ we have $m_r\to \infty.$ This proves the theorem.
\end{proof}

\subsection{Wojciechowski's anti-trees}
The second class of examples are so-called anti-trees introduced by Wojciechowski \cite{Wojciechowski11}. We call a graph an anti-tree if every vertex in $S_r$ is connected to all vertices in $S_{r+1}\cup S_{r-1}$ and to none in $S_r$. It was pointed out that $\sharp S_r=(r+1)^2$ is of particular interest in \cite{Bauer-Keller-Wojciechowski15}. In general case,  we consider $\sharp S_r=(r+1)^a$ for $a\in\mathbb N$, and we call it the anti-tree of order $a$, denoted by $G_a$.
\begin{example}\label{exam2}
  Let $G_a$ is an antitree of order $a\in \mathbb N.$ Then
 \begin{center}
\renewcommand\tabcolsep{28pt}
\newcommand{\zz}[2]{\rule[#1]{0pt}{#2}}
\begin{tabular}{|c|c|c|c|}
\hline\zz{-10pt}{28pt}

{\hfill $a$}                                 & {\hfill $a=1$}             &{\hfill $a=2$}          & {\hfill $a\geq 3$}              \\
\hline \zz{-10pt}{28pt}  $h(G)$             &     $2$                    & $4$                    &    $2^a$                         \\
\hline \zz{-10pt}{28pt}  $h_\infty(G)$      &     $2$                    &$\infty$                &      $\infty$                     \\
\hline \zz{-10pt}{28pt}  $h_M(G)$           &     $1$                   &$\frac{4}{\sqrt{10}}$    &    $\frac{2^a}{\sqrt{1+3^a}}$       \\
\hline \zz{-10pt}{28pt}  $h_{\infty,M}(G)$  &    $0$                    &  $\frac{3}{\sqrt{2}}$   &    $\infty$                     \\
\hline \zz{-10pt}{28pt}  $h_N(G)$           &     $0$                    &   $0$                  &      $0$             \\
\hline \zz{-10pt}{28pt}  $h_{\infty,N}(G)$    &     $0$                    &    $0$                 &       $0$            \\
\hline
\end{tabular}
\end{center}
\end{example}
\begin{proof}
  For any $a\in \N$ and $r\in \Z_{\geq 0}$, we consider three cases as in Definition \ref{n:notation}:
  \begin{enumerate}[1.]
    \item For Cheeger constants for physical Laplacians, $\mu_r=(r+1)^a(r+2)^a$ and $\nu_r=(r+1)^a.$
    \item  For the modified Cheeger constant for physical Laplacians, $\nu_r=(r+1)^a$ and
        $\textrm{Deg}(r)=\frac{\deg(r)}{\nu_r}=\frac{r^a(r+1)^a+(r+1)^a(r+2)^a}{(r+1)^a}=r^a+(r+2)^a$.
        \[\rho'(r,r+1)=\min\left\{\frac{1}{\sqrt{\textrm{Deg}(r)}}, \frac{1}{\sqrt{\textrm{Deg}(r+1)}}\right\}=\frac{1}{\sqrt{(r+1)^a+(r+3)^a}}.\]
         Hence, $\mu'_r=\mu_r\rho'(r,r+1)=\frac{(r+1)^a(r+2)^a}{\sqrt{(r+1)^a+(r+3)^a}}$ . 
    \item For the normalized Cheeger constant of normalized Laplacians, $\mu_r=(r+1)^a(r+2)^a$ and $\nu_r=r^a(r+1)^a+(r+1)^a(r+2)^a.$
    \end{enumerate}
    Using Theorem \ref{t:spherically symmetric}, Lemma \ref{l:calculus lemma1} and Lemma \ref{l:calculus lemma2}, we prove the results by the basic calculus.
    \begin{enumerate}[1.]
      \item For Cheeger constants for physical Laplacians, $|\partial \underline{B}_r|=\mu_r=(r+1)^a(r+2)^a,$
      $|\underline{B}_r|=\sum_{i=0}^r(i+1)^a$ and $|\underline{S}_r|=(r+1)^a.$
           Let
          \[f(r)=\frac{|\partial \underline{B}_r|}{|\underline{B}_r|}=\frac{(r+1)^a(r+2)^a}{\sum_{i=0}^r(i+1)^a}.\]

          For $a=1$, $f(r)=\frac{(r+1)(r+2)}{\sum_{i=0}^r(i+1)}=2=2^a$.

          For $a\geq 2$, since
          \begin{align}\begin{split}\label{mono-f(r)}
          f(r)\geq \frac{(r+1)^a(r+2)^a}{\sum_{i=0}^r(i+1)^a}
           \geq \frac{(r+1)^a(r+2)^a}{(r+1)^{a+1}}
          \geq (r+2)^{a-1},
          \end{split}\end{align}
          $f(r)\geq f(0)=2^a$ for $r\geq 2$. Hence, $\inf\limits_{r\geq 0}f(r)=\min\{f(1), f(0)\}=2^a$.
          By Theorem \ref{t:spherically symmetric}, we have
          \begin{eqnarray*}
            h(G)=\inf_{r\geq 0}\frac{|\partial \underline{B}_r|}{|\underline{B}_r|}= \inf_{r\geq 0}f(r)=2^a.
          \end{eqnarray*}
          Using Lemma \ref{l:calculus lemma2}, Lemma \ref{l:calculus lemma1}, and Theorem \ref{t:spherically symmetric}, we have
          \begin{align*}
          h_{\infty}(G)&=\lim_{r\to\infty}\frac{|\partial \underline{B}_r|-|\partial \underline{B}_{r-1}|}{|\underline{S}_r|}\\
          &=\lim_{r\to\infty}\frac{(r+1)^a(r+2)^a-r^a(r+1)^a}{(r+1)^a}\\ &=\begin{cases}
          2,\ \ &a=1 \\
          \infty, &a\geq 2.
          \end{cases}
          \end{align*}
       \item For modified Cheeger constants of physical Laplacians, $|\underline{S}_r|=(r+1)^a,$
        \[|\partial \underline{B}_r|=\mu_r'=\frac{(r+1)^a(r+2)^a}{\sqrt{(r+1)^a+(r+3)^a}},\ \ |\underline{B}_r|=\sum_{i=0}^r(i+1)^a.\] Take
        \[f(r):=\frac{|\partial \underline{B}_r|}{|\underline{B}_r|}=\frac{(r+1)^a(r+2)^a}{\sum_{i=0}^r(i+1)^a\sqrt{(r+1)^a+(r+3)^a}}.\]

     For $a=1$, $f(r)=\frac{(r+1)(r+2)}{\sum_{i=0}^r(i+1)\sqrt{(r+1)+(r+3)}}=\frac{2}{\sqrt{2r+4}}$, so $\inf\limits_{r\geq 0}f(r)=0$.

     For $a=2$, by $\sum_{i=0}^r(i+1)^2=\frac{1}{6}(r+1)(r+2)(2r+3)$, we obtain
     \[
     f(r)=\frac{6(r+1)(r+2)}{(2r+3)\sqrt{(r+1)^2+(r+3)^2}}.
     \]
     Let $y=\frac{r+2}{r+1}$. Then $r\geq 0$ gives $1<y\leq 2$. Note that
     $h(y)=\frac{6y}{(y+1)\sqrt{1+(2y-1)^2}},$
     and
     \[\frac{d}{dy}h(y)=-\frac{3 \sqrt{2} \left(-1 + y - y^2 + 2 y^3\right)}{(1 + y)^2 \left(1 - 2 y + 2 y^2\right)^\frac{3}{2}}<0.\]
     Hence, $f(r)=h(y)\geq h(2)=f(0)$ for $r\geq 0$.

     For $a=3$, by $\sum_{i=0}^r(i+1)^3=\frac{1}{4}(r+1)^2(r+2)^2$, $f(r)=\frac{4(r+1)(r+2)}{\sqrt{(r+1)^3+(r+3)^3}}$,
     \begin{align*}
     f'(r)&=\frac{\sqrt{2}\left(27+ 25r+7r^2+r^3\right)}{\left(7+4r+ r^2\right)\sqrt{14 + 15 r + 6 r^2 + r^3}}>0.\\
     \end{align*}
     Hence, $f(r)\geq f(0)$ for $r\geq 0$.

     For $a\geq4$, $f(0)=\frac{2^a}{\sqrt{1+3^a}}\leq \frac{2^a}{\sqrt{3}^a}$ and \[f(r)\geq\frac{(r+1)^a(r+2)^a}{(r+1)^{a+1}\sqrt{2(r+3)^a}}
          \geq \frac{(r+2)^{a-1}}{\sqrt{2}(r+3)^{\frac{a}{2}}}.\]
     For $f(r)\geq f(0)$, it suffices to prove
     $\frac{(r+2)^{a-1}}{\sqrt{2}(r+3)^{\frac{a}{2}}}\geq \frac{2^a}{\sqrt{3}^a}$, that is,
     \begin{align}\begin{split}\label{equi-form}
     \left(\frac{r+2}{\sqrt{r+3}}\cdot \frac{\sqrt{3}}{2}\right)^a \frac{1}{r+2}\geq \sqrt{2}.
     \end{split}\end{align}
     Since $\frac{r+2}{\sqrt{r+3}}\cdot \frac{\sqrt{3}}{2}=\left(\sqrt{r+3}-\frac{1}{\sqrt{r+3}}\right)\frac{\sqrt{3}}{2}\geq \left(\sqrt{3}-\frac{1}{\sqrt{3}}\right)\frac{\sqrt{3}}{2}=1$ for $r\geq 0$, it suffices to prove that (\ref{equi-form}) holds for $a=4$. Let \[g(r)=\left(\frac{r+2}{\sqrt{r+3}}\cdot \frac{\sqrt{3}}{2}\right)^4 \frac{1}{r+2}=\frac{9}{16}\frac{(r+2)^3}{(r+3)^2}.\] Then $g(r)$ is a increasing function with $g(2)=\frac{36}{25}\geq \sqrt{2} $. Hence, (\ref{equi-form}) holds for $r\geq2$, which yields $f(r)\geq f(0)$ for $r\geq 2$. Note that
     \[
     f(1)
     =\frac{6^a}{(1+2^a)\sqrt{2^a+4^a}}
     \geq\frac{6^a}{2^{a+1}\sqrt{2}\cdot  2^a }  =\frac{1}{2\sqrt{2}}\left(\frac{3}{2}\right)^a\geq \left(\frac{2}{\sqrt{3}}\right)^a\geq f(0)
     \]
    holds for $a\geq 4$.
    Hence, $f(r)\geq f(0)$ for $r\geq 0$.

    Therefore, we obtain
        \begin{eqnarray*}
            h_M(G)=\inf_{r\geq 0}\frac{|\partial \underline{B}_r|}{|\underline{B}_r|}=\inf_{r\geq 0}f(r)=f(0)=\frac{2^a}{\sqrt{1+3^a}}.
          \end{eqnarray*}

        For $a=1$, $\lim\limits_{r\to\infty}f(r)=\lim\limits_{r\to \infty}\frac{1}{\sqrt{2r+1}}=0$.
        For $a=2$, \[\lim\limits_{r\to\infty}f(r)=\lim\limits_{r\to\infty}\frac{6(r+1)(r+2)}{(2r+3)\sqrt{(r+1)^2+(r+3)^2}}=\frac{3}{\sqrt{2}}.\]
        For $a\geq 3$, setting $h(r):=\frac{(r+2)^a}{\sqrt{(r+1)^a+(r+3)^a}}$, we have
        \begin{align*}
        \frac{|\partial \underline{B}_r|-|\partial \underline{B}_{r-1}|}{|\underline{S}_r|}
        &=\frac{(r+2)^a}{\sqrt{(r+1)^a+(r+3)^a}}-\frac{r^a}{\sqrt{r^a+(r+2)^a}}\\
        &=h(r)-h(r-1)+\frac{(r+1)^a-r^a}{\sqrt{r^a+(r+2)^a}}.
        \end{align*}
        By $h'(r)\geq 0$ and
        $\lim\limits_{r\to\infty}\frac{(r+1)^a-r^a}{\sqrt{r^a+(r+2)^a}}=\infty,$
        we obtain
        \[
        h_{\infty,M}(G)=\lim_{r\to\infty}\frac{|\partial \underline{B}_r|-|\partial \underline{B}_{r-1}|}{|\underline{S}_r|}=\infty.\]
       \item For the normalized Cheeger constants of normalized Laplacians,
       \[|\partial \underline{B}_r|=\mu_r=(r+1)^a(r+2)^a,\ \  |\underline{S}_r|=r^a(r+1)^a+(r+1)^a(r+2)^a,\]
       and
       \begin{align*}
       |\underline{B}_r|
       &=\sum_{i=0}^r\left[i^a(i+1)^a+(i+1)^a(i+2)^a\right]\\
       &=2\sum_{i=1}^ri^a(i+1)^a+(r+1)^a(r+2)^a.
       \end{align*}
       Let
       \[f(r)=\frac{|\partial \underline{B}_r|}{|\underline{B}_r|}=\frac{(r+1)^a(r+2)^a}{2\sum_{i=1}^ri^a(i+1)^a+(r+1)^a(r+2)^a}.\]
      Since
      \begin{align*}
      \sum_{i=1}^r(i+1)^ai^a\geq \sum_{i=2}^ri^{2a}\geq \sum_{i=2}^r\int_{i-1}^i x^{2a}\geq \int_1^r x^{2a}dx=\frac{r^{2a+1}-1}{2a+1},
      \end{align*}
      \[
      0<f(r)\leq \frac{2a+1}{2}\cdot\frac{(r+1)^a(r+2)^a}{r^{2a+1}-1+(r+1)^a(r+2)^a}\to 0(r\to \infty).
      \]
      So we obtain $\inf\limits_{r\geq 0}f(r)=0$.
      Hence, \[h_N(G)=\inf_{r\geq 0}\frac{|\partial \underline{B}_r|}{|\underline{B}_r|}=\inf\limits_{r\geq 0}f(r)=0.\]

        Moreover, it follows from $$\lim_{r\to\infty}\frac{|\partial \underline{B}_r|-|\partial \underline{B}_{r-1}|}{|\underline{S}_r|}=\lim_{r\to\infty}\frac{(r+1)^a(r+2)^a-r^a(r+1)^a}{r^a(r+1)^a+(r+1)^a(r+2)^a}=0$$ that
        $$h_{\infty,N}(G)=0.$$
    \end{enumerate}

\end{proof}

\textbf{Acknowledgements.} We thank Frank Bauer, Huabin Ge, and Wenfeng Jiang for many discussions and suggestions on $p$-Laplacian eigenvalues on graphs. B.H. is supported by NSFC, no.11831004 and no. 11826031. L. W. is supported by NSFC, no. 11671141.

\bibliographystyle{alpha}
\bibliography{p-laplacian-eigenvalue}

\begin{thebibliography}{LOGT12}

\bibitem[AC13]{andrewsclutterbuck}
B.~Andrews and J.~Clutterbuck.
\newblock Sharp modulus of continuity for parabolic equations on manifolds and
  lower bounds for the first eigenvalue.
\newblock {\em Anal. PDE}, 6(5):1013--1024, 2013.

\bibitem[AM85]{Alon-Milman85}
N.~Alon and V.~D. Milman.
\newblock {$\lambda_1,$} isoperimetric inequalities for graphs, and
  superconcentrators.
\newblock {\em J. Combin. Theory Ser. B}, 38(1):73--88, 1985.

\bibitem[Amg03]{Amghibech03}
S.~Amghibech.
\newblock Eigenvalues of the discrete {$p$}-{L}aplacian for graphs.
\newblock {\em Ars Combin.}, 67:283--302, 2003.

\bibitem[BHJ14]{BauerHuaJost14}
F.~Bauer, B.~B. Hua, and J.~Jost.
\newblock The dual {C}heeger constant and spectra of infinite graphs.
\newblock {\em Adv. Math.}, 251:147--194, 2014.

\bibitem[BK13]{Breuer-Keller13}
J.~Breuer and M.~Keller.
\newblock Spectral analysis of certain spherically homogeneous graphs.
\newblock {\em Oper. Matrices}, 7(4):825--847, 2013.

\bibitem[BKW15]{Bauer-Keller-Wojciechowski15}
F.~Bauer, M.~Keller, and K.~Wojciechowski.
\newblock Cheeger inequalities for unbounded graph {L}aplacians.
\newblock {\em J. Eur. Math. Soc.}, 17(2):259--271, 2015.

\bibitem[Bol13]{Bolla13}
M.~Bolla.
\newblock {\em Spectral clustering and biclustering}.
\newblock John Wiley and Sons, Ltd., Chichester, 2013.

\bibitem[CH53]{Courant-Hilbert53}
R.~Courant and D.~Hilbert.
\newblock {\em Methods of mathematical physics. {V}ol. {I}}.
\newblock Interscience Publishers, Inc., New York, N.Y., 1953.

\bibitem[Cha84]{Chavel84}
I.~Chavel.
\newblock {\em Eigenvalues in {R}iemannian geometry}, volume 115 of {\em Pure
  and Applied Mathematics}.
\newblock Academic Press, Inc., Orlando, FL, 1984.
\newblock Including a chapter by Burton Randol, With an appendix by Jozef
  Dodziuk.

\bibitem[Cha16]{Chang16}
K.~C. Chang.
\newblock Spectrum of the 1-{L}aplacian and {C}heeger's constant on graphs.
\newblock {\em J. Graph Theory}, 81(2):167--207, 2016.

\bibitem[Che70]{Cheeger70}
J.~Cheeger.
\newblock {\em A lower bound for the smallest eigenvalue of the {L}aplacian}.
\newblock Problems in analysis ({P}apers dedicated to {S}alomon {B}ochner,
  1969), Princeton Univ. Press, Princeton, N. J., 1970.
\newblock 195--199.

\bibitem[Chu97]{Chung97}
Fan R.~K. Chung.
\newblock {\em {Spectral Graph Theory (CBMS Regional Conference Series in
  Mathematics, No. 92)}}.
\newblock {American Mathematical Society}, 1997.

\bibitem[CO97]{ChengOden97}
S.-Y. Cheng and K.~Oden.
\newblock Isoperimetric inequalities and the gap between the first and second
  eigenvalues of an {E}uclidean domain.
\newblock {\em J. Geom. Anal.}, 7(2):217--239, 1997.

\bibitem[CSZ15]{Chang-Shao-Zhang15}
K.~C. Chang, S.~H. Shao, and D.~Zhang.
\newblock The 1-{L}aplacian {C}heeger cut: theory and algorithms.
\newblock {\em J. Comput. Math.}, 33(5):443--467, 2015.

\bibitem[CSZ17a]{Chang-Shao-Zhang17a}
K.~C. Chang, S.~H. Shao, and D.~Zhang.
\newblock Nodal domains of eigenvectors for 1-{L}aplacian on graphs.
\newblock {\em Adv. Math.}, 308:529--574, 2017.

\bibitem[CSZ17b]{Chang-Shao-Zhang17}
K.C. Chang, S.~H. Shao, and D.~Zhang.
\newblock Cheeger's cut, maxcut and the spectral theory of $1$-laplacian on
  graphs.
\newblock {\em Science China Mathematics}, 60(11):1963--1980, 2017.

\bibitem[DH73]{Donath-Hoffman73}
W.E. Donath and A.J. Hoffman.
\newblock Lower bounds for the partitioning of graphs.
\newblock {\em IBM J. Res. Develop.}, 17:420--425, 1973.

\bibitem[DK86]{Dodziuk-Kendall86}
J.~Dodziuk and W.~S. Kendall.
\newblock Combinatorial {L}aplacians and isoperimetric inequality.
\newblock In {\em From local times to global geometry, control and physics
  ({C}oventry, 1984/85)}, volume 150 of {\em Pitman Res. Notes Math. Ser.},
  pages 68--74. Longman Sci. Tech., Harlow, 1986.

\bibitem[Dod84]{Dodziuk84}
J.~Dodziuk.
\newblock Difference equations, isoperimetric inequality and transience of
  certain random walks.
\newblock {\em Trans. Amer. Math. Soc.}, 284:787--794, 1984.

\bibitem[FF60]{FedererFleming60}
H.~Federer and W.H. Fleming.
\newblock Normal and integral currents.
\newblock {\em Ann. of Math. (2)}, 72:458--520, 1960.

\bibitem[Fuj96]{Fujiwara96}
K.~Fujiwara.
\newblock The {L}aplacian on rapidly branching trees.
\newblock {\em Duke Math. J.}, 83(1):191--202, 1996.

\bibitem[Gri09]{Grigor09}
A.~Grigor'yan.
\newblock {\em Analysis on Graphs}.
\newblock Lecture Notes, University Bielefeld, 2009.

\bibitem[HB10]{Hein-Buhler10}
M.~Hein and T.~B\"{u}hler.
\newblock An inverse power method for nonlinear eigenproblems with applications
  in $1$-spectral clustering and sparse pca.
\newblock {\em Advances in Neural Information Processing Systems (NIPS)}, pages
  847--855, 2010.
\newblock MIT Press, Cambridge, MA.

\bibitem[KC10]{Kim-Chung10}
J.-H. Kim and S.-Y. Chung.
\newblock Comparison principles for the {$p$}-{L}aplacian on nonlinear
  networks.
\newblock {\em J. Difference Equ. Appl.}, 16(10):1151--1163, 2010.

\bibitem[KF03]{Kawohl-Fridman03}
B.~Kawohl and V.~Fridman.
\newblock Isoperimetric estimates for the first eigenvalue of the
  {$p$}-{L}aplace operator and the {C}heeger constant.
\newblock {\em Comment. Math. Univ. Carolin.}, 44(4):659--667, 2003.

\bibitem[KL06]{Kawohl-Lindqvist06}
B.~Kawohl and P.~Lindqvist.
\newblock Positive eigenfunctions for the {$p$}-{L}aplace operator revisited.
\newblock {\em Analysis (Munich)}, 26(4):545--550, 2006.

\bibitem[KLW13]{Keller-Lenz-Wojciechowski13}
M.~Keller, D.~Lenz, and R.K. Wojciechowski.
\newblock Volume growth, spectrum and stochastic completeness of infinite
  graphs.
\newblock {\em Math. Z.}, 274(3-4):905--932, 2013.

\bibitem[KM16]{Keller-Mugnolo16}
M.~Keller and D.~Mugnolo.
\newblock General {C}heeger inequalities for {$p$}-{L}aplacians on graphs.
\newblock {\em Nonlinear Anal.}, 147:80--95, 2016.

\bibitem[Li12]{Li12}
P.~Li.
\newblock {\em Geometric analysis}, volume 134 of {\em Cambridge Studies in
  Advanced Mathematics}.
\newblock Cambridge University Press, Cambridge, 2012.

\bibitem[Lin93]{Lindqvist93}
P.~Lindqvist.
\newblock On nonlinear {R}ayleigh quotients.
\newblock {\em Potential Anal.}, 2(3):199--218, 1993.

\bibitem[Liu15]{Liu15}
S.~P. Liu.
\newblock Multi-way dual {C}heeger constants and spectral bounds of graphs.
\newblock {\em Adv. Math.}, 268:306--338, 2015.

\bibitem[LOGT12]{LeeGharanTrevisan12}
J.~R. Lee, S.~Oveis~Gharan, and L.~Trevisan.
\newblock Multi-way spectral partitioning and higher-order {C}heeger
  inequalities.
\newblock In {\em S{TOC}'12---{P}roceedings of the 2012 {ACM} {S}ymposium on
  {T}heory of {C}omputing}, pages 1117--1130. ACM, New York, 2012.

\bibitem[LT77]{Lindenstrauss-Tzafriri77}
J.~Lindenstrauss and L.~Tzafriri.
\newblock {\em Classical {B}anach spaces. {I}}.
\newblock Springer-Verlag, Berlin-New York, 1977.
\newblock Sequence spaces, Ergebnisse der Mathematik und ihrer Grenzgebiete,
  Vol. 92.

\bibitem[Lub94]{Lubotzky94}
A.~Lubotzky.
\newblock {\em Discrete groups, expanding graphs and invariant measures},
  volume 125 of {\em Progress in Mathematics}.
\newblock Birkh\"{a}user Verlag, Basel, 1994.
\newblock With an appendix by Jonathan D. Rogawski.

\bibitem[Mat00]{Matei00}
Ana-M. Matei.
\newblock First eigenvalue for the {$p$}-{L}aplace operator.
\newblock {\em Nonlinear Anal.}, 39(8, Ser. A: Theory Methods):1051--1068,
  2000.

\bibitem[NJW]{Ng-Jordan-Weiss2001}
A.Y. Ng, M.I. Jordan, and Y.~Weiss.
\newblock {\em On spectral clustering--analysis and an algorithm}.
\newblock in: Advances in Neural Information Processing Systems (NIPS) 14, pp.
  849--856, MIT Press, Cambridge, MA, 2001.

\bibitem[NV14]{NaberValtorta14}
A.~Naber and D.~Valtorta.
\newblock Sharp estimates on the first eigenvalue of the {$p$}-{L}aplacian with
  negative {R}icci lower bound.
\newblock {\em Math. Z.}, 277(3-4):867--891, 2014.

\bibitem[Par11]{Park11}
J.-H. Park.
\newblock On a resonance problem with the discrete {$p$}-{L}aplacian on finite
  graphs.
\newblock {\em Nonlinear Anal.}, 74(17):6662--6675, 2011.

\bibitem[PC11]{Park-Chung11}
J.-H. Park and S.-Y. Chung.
\newblock Positive solutions for discrete boundary value problems involving the
  {$p$}-{L}aplacian with potential terms.
\newblock {\em Comput. Math. Appl.}, 61(1):17--29, 2011.

\bibitem[RS78]{Reed-Simon78}
M.~Reed and B.~Simon.
\newblock {\em Methods of modern mathematical physics. {IV}. {A}nalysis of
  operators}.
\newblock Academic Press [Harcourt Brace Jovanovich, Publishers], New
  York-London, 1978.

\bibitem[SW17]{SetoWei17}
S.~Seto and G.~F. Wei.
\newblock First eigenvalue of the {$p$}-{L}aplacian under integral curvature
  condition.
\newblock {\em Nonlinear Anal.}, 163:60--70, 2017.

\bibitem[SY94]{SchoenYau94}
R.~Schoen and S.-T. Yau.
\newblock {\em Lectures on differential geometry}.
\newblock Conference Proceedings and Lecture Notes in Geometry and Topology, I.
  International Press, Cambridge, MA, 1994.

\bibitem[Tak03]{Takeuchi03}
H.~Takeuchi.
\newblock The spectrum of the {$p$}-{L}aplacian and {$p$}-harmonic morphisms on
  graphs.
\newblock {\em Illinois J. Math.}, 47(3):939--955, 2003.

\bibitem[TH18]{Tudisco-Hein18}
F.~Tudisco and M.~Hein.
\newblock A nodal domain theorem and a higher-order {C}heeger inequality for
  the graph {$p$}-{L}aplacian.
\newblock {\em J. Spectr. Theory}, 8(3):883--908, 2018.

\bibitem[Val12]{Valtorta12}
D.~Valtorta.
\newblock Sharp estimate on the first eigenvalue of the {$p$}-{L}aplacian.
\newblock {\em Nonlinear Anal.}, 75(13):4974--4994, 2012.

\bibitem[Wan12]{Wang12}
L.~F. Wang.
\newblock Eigenvalue estimate for the weighted {$p$}-{L}aplacian.
\newblock {\em Ann. Mat. Pura Appl. (4)}, 191(3):539--550, 2012.

\bibitem[Woj11]{Wojciechowski11}
R.~K. Wojciechowski.
\newblock {\em Stochastically incomplete manifolds and graphs, Random walks,
  boundaries and spectra}.
\newblock Birkh\"{a}user/Springer Basel AG, Basel, 2011.
\newblock (Progress in Probab., vol. 64,163--179).

\bibitem[Yam79]{Yamasaki79}
M.~Yamasaki.
\newblock Discrete potentials on an infinite network.
\newblock {\em Memoirs of the Faculty of Literature and Science, Shimane
  University}, 13:31--44, 1979.

\bibitem[Yau75]{Yau75}
S.-T. Yau.
\newblock Isoperimetric constants and the first eigenvalue of a compact
  {R}iemannian manifold.
\newblock {\em Ann. Sci. \'{E}cole Norm. Sup. (4)}, 8(4):487--507, 1975.

\end{thebibliography}

%
\end{document}